\documentclass[12pt,leqno,a4paper]{article}
\usepackage[T1]{fontenc}
\usepackage{color}

\usepackage[pdfborder={0 0 0}]{hyperref}

\usepackage{amssymb}
\usepackage{amsmath}
\usepackage{amsthm}
\usepackage{mathrsfs}

\usepackage[shortlabels]{enumitem}

\newtheorem{theorem}{Theorem}
\numberwithin{theorem}{section}
\newtheorem{proposition}[theorem]{Proposition}
\newtheorem{lemma}[theorem]{Lemma}
\newtheorem{corollary}[theorem]{Corollary}

\theoremstyle{remark}

\theoremstyle{definition}
\newtheorem{remark}[theorem]{Remark}
\newtheorem{definition}[theorem]{Definition}
\newtheorem{example}[theorem]{Example}

\numberwithin{equation}{section}

\newcommand\set[1]{\left\{\,#1\,\right\}}		
\newcommand\abs[1]{\left|#1\right|}				
\newcommand\ska[1]{\left\langle#1\right\rangle} 
\newcommand\norm[1]{\left\Vert#1\right\Vert}	

\DeclareMathOperator{\dist}{dist}				
\DeclareMathOperator{\sign}{sign}				
\DeclareMathOperator{\id}{id}					
\DeclareMathOperator{\tr}{tr}					
\DeclareMathOperator{\divv}{div}				

\def\N{\mathbb{N}}

\def\R{\mathbb{R}}

\newcommand{\cA}{{\mathcal A}}

\newcommand{\cC}{{\mathcal C}}

\newcommand{\cE}{{\mathcal E}}
\newcommand{\cF}{{\mathcal F}}

\newcommand{\cH}{{\mathcal H}}

\newcommand{\cS}{{\mathcal S}}

\newcommand{\cX}{{\mathcal X}}

\newcommand{\lamax}{\lambda_{\text{max}}}
\newcommand{\lamin}{\lambda_{\text{min}}}						

\usepackage[numbers]{natbib}

\begin{document}
\title{Relaxation of the Boussinesq system and applications to the Rayleigh-Taylor instability}
\author{Bj\"orn Gebhard \and J\'ozsef J. Kolumb\'an}
\date{}
\maketitle

\begin{abstract}
We consider the evolution of two incompressible fluids with homogeneous densities $\rho_-<\rho_+$ subject to gravity described by the inviscid Boussinesq equations and provide the explicit relaxation of the associated differential inclusion. 
The existence of a subsolution to the relaxation allows one to conclude the existence of turbulently mixing solutions to the original Boussinesq system.
As a specific application we investigate subsolutions emanating from the classical Rayleigh-Taylor initial configuration where the two fluids are separated by a horizontal interface with the heavier fluid being on top of the lighter. It turns out that among all self-similar subsolutions the criterion of maximal initial energy dissipation selects a linear density profile and a quadratic growth of the mixing zone. The subsolution selected this way can be extended in an admissible way to exist for all times. We provide two possible extensions with different long-time limits. The first one corresponds to a total mixture of the two fluids, the second corresponds to a full separation with the lighter fluid on top of the heavier. There is no motion in either of the limit states. 
\end{abstract}
%
%

\section{Introduction}

We investigate two incompressible fluids with homogeneous densities $0<\rho_-<\rho_+$ under the influence of gravity modelled by the Euler equations in Boussinesq approximation
\begin{align}\label{eq:bou}
\begin{split}
\partial_t v +\divv (v\otimes v) +\nabla p&=-\rho gA e_n,\\
\divv v&=0,\\
\partial_t \rho + \divv (\rho v)&=0.
\end{split}
\end{align}
The equations are considered on a bounded domain $\Omega\subset\R^n$ and a time interval $[0,T)$, $T>0$. The function $\rho:\Omega\times[0,T)\rightarrow\R$ is the normalized fluid density, i.e. $\rho\in\{\pm 1\}$ a.e., $v:\Omega\times[0,T)\rightarrow\R^n$ is the velocity field and $p:\Omega\times[0,T)\rightarrow\R$ the pressure of the fluid. Furthermore, $e_n\in\R^n$ denotes the $n$th coordinate vector, $g>0$ the gravitational constant and 
\[
A:=\frac{\rho_+-\rho_-}{\rho_++\rho_-}
\]
is the Atwood number. The incompressibility condition is complemented by the no-penetration boundary condition
\begin{equation}\label{eq:boundary_condition}
v\cdot \nu =0 \quad \text{on }\partial\Omega\times[0,T),
\end{equation}
where $\nu$ denotes the exterior unit normal of the boundary of $\Omega$, which is assumed to be sufficiently smooth. We will mostly  consider \eqref{eq:bou}, \eqref{eq:boundary_condition} with the unstable interface as initial data, i.e.,
\begin{align}\label{eq:initial_data}
\rho(x,0)=\sign(x_n),\quad v(x,0)=0,\quad x\in\Omega.
\end{align}

This initial data is a classical instance of the Rayleigh-Taylor instability which occurs whenever a lighter fluid or gas is accelerated into a heavier one -- a situation which appears in various research areas and applications, see \cite{Abarzhi,Boffetta,Zhou1,Zhou2} for an overview. Its linear (in)stability analysis goes back to Rayleigh \cite{Rayleigh} and Taylor \cite{Taylor}.

\subsection{General aspects of the Boussinesq system}

System \eqref{eq:bou} arises from the actual inhomogeneous incompressible Euler equations
\begin{align}\label{eq:euler_equations}
\begin{split}
\partial_t (\tilde{\rho} v) +\divv (\tilde{\rho}v\otimes v) +\nabla \tilde{p}&=-\tilde{\rho} g e_n,\\
\divv v&=0,\\
\partial_t \tilde{\rho} + \divv (\tilde{\rho} v)&=0,
\end{split}
\end{align}
via the normalization $\tilde{\rho}=\frac{1}{2}(\rho_++\rho_-)+\frac{1}{2}(\rho_+-\rho_-)\rho$, such that $\rho\in\set{\pm 1}$, and by the Boussinesq approximation, i.e., by neglecting the density difference $\rho_+-\rho_-$ in the acceleration term on the left-hand side of the first equation. The Boussinesq approximation therefore is only applicable in the regime of small Atwood number $A\ll 1$. 

In the present paper we consider the inviscid and indiffusive Boussinesq system \eqref{eq:bou}. More generally one can also add different sorts of diffusion terms in the momentum balance and/or the mass balance. In dimension $2$ global well-posedness results of sufficiently regular solutions and for different types of diffusion terms have been established in \cite{Chae,Danchin-Paicu,Hmidi,Hou-Li}, while finite time singularity formation for equation \eqref{eq:bou}, i.e. without any diffusive terms, has recently been shown in \cite{Elgindi}.

Local well-posedness statements for \eqref{eq:bou} considered in different sufficiently regular classes can be found in \cite{Chae-Nam,Danchin,Elgindi}. Note however that the horizontal interface \eqref{eq:initial_data} does not belong to these classes. On the other hand if one adds the diffusion terms $-\nu\Delta v$, $\nu>0$ and $-\mu\Delta \rho$, $\mu>0$ to the equations, local well-posedness has also been established for $L^p$ initial data \cite{Cannon-DiBenedetto}.

Contrary to local well-posedness the articles \cite{Bronzi,Chiodaroli} address the question of non-uniqueness of solutions. More precisely, \cite{Bronzi} shows the existence of wild solutions for \eqref{eq:bou} considered with $g=0$, i.e. for the homogeneous Euler equations augmented by a transport equation for a passive tracer. In \cite{Chiodaroli} the effect of the Coriolis force and a diffusion term in the continuity equation is added to \eqref{eq:bou} and non-uniqueness of weak solutions to given initial data is proven. Moreover, the non-uniqueness of admissible weak solutions is also shown in \cite{Chiodaroli} by construction of a suitable initial velocity field $v_0$ to a given initial density $\rho_0\in L^\infty(\Omega)\cap \cC^2(\Omega)$. The non-uniqueness results \cite{Bronzi,Chiodaroli} both rely on the method of convex integration introduced by De Lellis and Sz\'ekelyhidi to the context of fluid dynamics \cite{DeL-Sz-Annals,DeL-Sz-Adm}.

\subsection{Heuristic outline of results}
In this article we also address the inviscid Boussinesq system by means of convex integration, but our main goal here, as in \cite{GKSz} for the Euler equations \eqref{eq:euler_equations}, is to investigate nonlinear instability aspects of the Rayleigh-Taylor configuration \eqref{eq:initial_data} by providing the existence of solutions to the problem \eqref{eq:bou}, \eqref{eq:boundary_condition}, \eqref{eq:initial_data} that reflect a turbulent mixing. The oscillatory behaviour of solutions obtained by convex integration has also been utilized as an instance of turbulent mixing in the context of the Kelvin-Helmholtz instability \cite{Mengual-Sz-KH,Sz-KH} and the Muskat problem for the incompressible porous media equation \cite{Castro-Cordoba-Faraco,Castro-Faraco-Mengual,Cordoba-Faraco-Gancedo,Foerster-Sz,Hitruhin-Lindberg,Mengual,Noisette-Sz,Sz-Muskat}. Compared to \cite{Bronzi,Chiodaroli} this requires the explicit knowledge of the relaxation associated with \eqref{eq:bou}, which will be given here.

Using this relaxation we construct solutions to \eqref{eq:bou}, \eqref{eq:boundary_condition}, \eqref{eq:initial_data} considered on an $n$-dimensional quader $\Omega=(0,1)^{n-1}\times (-L,L)$, $n\geq 2$, which at any time $t>0$ with $\frac{1}{3}gAt^2\leq L$ are turbulently mixing in the space region $\mathscr{U}(t):=\set{x\in \Omega:\abs{x_n}<\frac{1}{3}gAt^2}$. The solutions all have a underlying self-similar subsolution in common, whose $\rho$-component is linear inside the mixing zone, i.e. $\rho_{sub}(x,t)=\frac{3x_n}{gAt^2}$ for $x\in \mathscr{U}(t)$.

The subsolution and hence the growth rate of the mixing zone $\frac{1}{3}gAt^2$ is selected uniquely and independently of the dimension by asking for maximal initial energy dissipation among all self-similar subsolutions. In particular the induced solutions are admissible with respect to the initial energy. The usage of maximal energy dissipation is motivated by the entropy rate admissibility criterion for hyperbolic conservation laws \cite{Dafermos} and has been investigated in \cite{Mengual-Sz-KH,Sz-KH} in the context of Euler subsolutions emanating from vortex-sheet initial data, as well as in \cite{Chiodaroli-Kreml,Feireisl} for compressible Euler systems. As in \cite{Mengual-Sz-KH} we focus here on maximal initial dissipation, i.e. the selection applies to small times.

Beyond small times, we show that the subsolutions can be extended in an admissible way past the time where the mixing zone hits the boundary. In fact we provide two possible extensions which in the long-time limit converge to two different states of the fluid: the fully mixed, isotropic state without any turbulent motion, this can be seen as the two different fluids forming now a single homogeneous fluid at rest, and the demixed stationary configuration where the two fluids are also at rest, but completly separated with the heavier fluid below the lighter. 
In other words this shows the existence of turbulent heteroclinic solutions emanating from the unstable interface configuration.

\subsection{Brief comparison to experiments}
There are numerous works addressing the Rayleigh-Taylor instability in different settings by means of experiments, numerical simulations and theoretical investigations for reduced models. For further reading we simply refer to the references given in the reviews \cite{Abarzhi,Boffetta,Zhou1,Zhou2}. At this point we only like to quickly compare our solutions for \eqref{eq:bou} with the results of the experiments carried out in \cite{Ramaprabhu-Andrews} at Atwood number $A\sim 7.5\cdot 10^{-4}$.  

First of all both the experiments and our solutions have a growth rate for the mixing zone like $\alpha gAt^2$. While the criterion of maximal initial energy dissipation selects $\alpha=\frac{1}{3}$ for our solutions, the actual constant observed in \cite{Ramaprabhu-Andrews} is $\alpha=0.07$. Concerning self-similarity it is written in \cite{Ramaprabhu-Andrews}: ``The saturation of $\alpha$ at late time to a constant value of 0.07 suggests that the flow
reaches self-similarity in these experiments.''
Another quantity we can easily compare is the ratio between dissipated energy and released potential energy. We will see in Section \ref{sec:subsolutions} (Remark \ref{rem:energy_ration}) that up to an arbitrary small error our solutions show a ratio of $\frac{D}{P_{rel.}}=\frac{1}{3}$, while the ratio measured in \cite{Ramaprabhu-Andrews} in dimension $3$ is $\frac{D}{P_{rel.}}=0.49$.

We conclude that the solutions selected by maximal initial energy dissipation stand the comparison to actual experiments on a qualitative level, but it remains the interesting question if the gap between $\alpha=\frac{1}{3}$ vs. $\alpha=0.07$ and $\frac{D}{P_{rel.}}=\frac{1}{3}$ vs. $\frac{D}{P_{rel.}}=0.49$ can be improved in the future. In particular, it would be an interesting open problem to see if the measured values correspond perhaps to the optimization of some other mathematical quantity (other than the initial energy dissipation).

\subsection{The role of the energy as a prescribed quantity}

As in some other previous works of convex integration in fluid mechanics (e.g. \cite{DeL-Sz-Annals,DeL-Sz-Adm,GKSz}), there is a microscopic quantity which one has to prescribe in a continuous way in order to implement the convex integration. For instance, in the case of the homogeneous density incompressible Euler equations, this quantity was the kinetic energy $\frac{1}{2}|v|^2$;
respectively in the case of the inhomogeneous incompressible Euler equations in \cite{GKSz} it was the quantity $\frac{1}{2}\rho|v+gt e_n|^2$, which corresponded to the kinetic energy of a transformed system,
and which can be seen as the kinetic energy of the original system plus a linear function of the momentum and the density.
In our case the appropriate prescribed quantity which gives rise to the subsolutions mentioned before is $\frac{1}{2}|v|^2+\frac{1}{3}\rho gAx_n$, i.e. the kinetic energy plus a fraction of the potential energy of the system. 

In fact, 
both our convex integration strategy and the construction of our subsolutions from Section \ref{sec:subsolutions}
can be carried out while prescribing the quantity $\frac{1}{2}|v|^2+\epsilon \rho gAx_n$, for any $\epsilon\in[0,1]$. However, setting for instance $\epsilon=1$, i.e. prescribing the total energy of the system, leads to solutions which are not admissible. The value  $\epsilon=\frac{1}{3}$
is obtained by the process of maximizing the initial energy dissipation.

For further details, see the more detailed discussion in Sections \ref{sec:statement_relaxation}, \ref{sec:choices_concerning_the_energy} and the construction in Section \ref{sec:subsolutions}.

\subsection{Comparison to \texorpdfstring{\cite{GKSz}}{the inhomogeneous Euler equation}}

In \cite{GKSz} the authors together with L. Sz\'ekelyhidi have addressed the Rayleigh-Taylor instability for the inhomogeneous incompressible Euler equations \eqref{eq:euler_equations}. In the aforementioned paper we obtained the existence of admissible turbulently mixing solutions for a sufficiently high density ratio $\frac{\rho_+}{\rho_-}$, which translates to the Atwood number $A$ being in the so-called ``ultra high'' range, $A\geq 0.845$, i.e. far away from the Boussinesq range.

The proof there also relied on the explicit computation of the relaxation and convex integration within the Tartar framework. The computations for the convex hull in Section \ref{sec:convex_hull} resemble the computations done in \cite{GKSz}. While in the aforementioned paper a transformation of \eqref{eq:euler_equations} onto an accelerated domain could be used in order to fit the system exactly into the Tartar framework (by which we mean that the gravity term in the momentum equation disappeared), here we can no longer use this transformation due to the Boussinesq approximation, and instead construct localized plane waves for an inhomogeneous linear system, see Section \ref{sec:waves}.

However, the main difference to \cite{GKSz} is the way subsolutions are constructed and selected. In \cite{GKSz} we reduced the relaxed system to a conservation law, which has some similarities to the conservation law appearing in \cite{Otto} in a different approach to relax  the incompressible porous media equation, and picked the unique entropy solution as our subsolution density profile. The subsolution found this way is self-similar and admissible for big enough $A$.
Here instead, we consider the whole zoo of self-similar subsolutions and set up a variational problem whose unique minimizer corresponds to the subsolution maximizing the initial energy dissipation.

As illustrated in the previous subsection, contrary to \cite{GKSz} strictly speaking we do not provide one relaxation of the nonlinear system, but several relaxations, which differ in the amount of allowed turbulent behaviour in the local energy density, see Section \ref{sec:choices_concerning_the_energy}. 

The discussion of possible long time limits in Section \ref{sec:beyond_small_times} for the subsolution is not part of \cite{GKSz}.

\subsection{Outline of the paper}
In Section \ref{sec:statements} we formulate our results concerning the relaxation of \eqref{eq:bou} and the investigation of subsolutions in a precise way. Section \ref{sec:tartar} contains the steps needed to carry out convex integration in the Tartar framework and Section \ref{sec:subsolutions} contains the construction and selection of self-similar subsolutions.

\section{Statement of results}\label{sec:statements}
Let $\Omega\subset \R^n$ be a bounded domain and $T>0$. Our notion of solution to system \eqref{eq:bou}, \eqref{eq:boundary_condition} on $\Omega\times[0,T)$ for general initial data $\rho(\cdot,0)=\rho_0$, $v(\cdot,0)=v_0$ with
\begin{align}\label{eq:general_initial_data}
\rho_0\in L^\infty(\Omega),\quad \rho_0\in\{\pm1\}~a.e., \quad v_0\in L^2(\Omega;\R^n),\quad \divv v_0=0\text{ weakly},
\end{align}
is as follows.
\begin{definition}[Weak solutions]\label{def:weaksols}
Let $(\rho_0,v_0)$ be as in \eqref{eq:general_initial_data}.
We say that $(\rho,v)\in L^\infty(\Omega\times(0,T))\times L^2(\Omega\times(0,T);\R^n)$ is a weak solution to \eqref{eq:bou}, \eqref{eq:boundary_condition} with initial data $(\rho_0,v_0)$ if for any test functions $\Phi\in C^\infty_c(\Omega\times[0,T);\mathbb{R}^n) $, $\Psi\in C^\infty_c(\overline{\Omega}\times[0,T)) $, such that $\Phi$ is divergence-free, we have
\begin{align*}
\int_0^T\int_{\Omega} \left[v \cdot \partial_t \Phi + \langle v\otimes v ,\nabla\Phi\rangle - gA \rho \Phi_n \right] \ dx \ dt +\int_{\Omega}v_0 (x)\cdot \Phi(x,0)\ dx=0,\\
\int_0^T\int_\Omega v\cdot\nabla \Psi\:dx\:dt=0,\\
\int_0^T\int_{\Omega} \left[\rho \partial_t \Psi + \rho v\cdot\nabla\Psi \right] \ dx \ dt +\int_{\Omega} \rho_0 (x) \Psi(x,0)\ dx=0,
\end{align*}
and if $\rho(x,t)\in\set{\pm 1}$ for a.e. $(x,t)\in\Omega\times(0,T)$.
\end{definition}

Observe that the definition of $v$ being weakly divergence-free includes the no-flux boundary condition. Moreover, for a smooth vectorfield $v$ the condition $\rho\in\{\pm 1\}$ automatically holds true, because then the density is transported along the flow associated with $v$, but for weaker notions of solutions this property in general is lost, see for example \cite{Modena}. 
Furthermore, a (in general distributional) pressure $p$ can be recovered from $(\rho,v)$ as in the case of the homogeneous Euler equations, see \cite{Temam}.

The local energy density function $\mathcal E\in L^1(\Omega\times(0,T))$ associate with a weak solution $(\rho,v)$ reads
\begin{align}\label{eq:local_energy_density}
\mathcal E(x,t):=\frac{1}{2}\abs{v(x,t)}^2+\rho(x,t)gAx_n.
\end{align} 
Indeed, testing a sufficiently smooth solution of \eqref{eq:bou} with $v$ one sees that the total energy $\int_{\Omega}\mathcal E(x,t)\:dx$ is independent of $t$. However, this property in general fails to be true for weak solutions of Euler type equations, see \cite{DeL-Sz-Annals} for Euler and \cite{Chiodaroli} for the Boussinesq system. In order to rule out unphysical solutions due to an increase in energy and in view of the weak-strong uniqueness principle in various equations in fluid dynamics \cite{Wiedemann} we require the solutions to satisfy the following admissibility condition.

\begin{definition}[Admissible weak solutions]\label{def:admissibility_of_weak_solutions}
A weak solution $(\rho,v)$ in the sense of Definition \ref{def:weaksols} is called admissible provided it satisfies the weak energy inequality
\[
\int_\Omega \mathcal E(x,t)\:dx\leq \int_\Omega \frac{1}{2}\abs{v_0(x)}^2+\rho_0(x)gAx_n\:dx\text{ for a.e. }t\in(0,T).
\]
\end{definition}

\subsection{The relaxation}\label{sec:statement_relaxation}
Next we will reformulate equation \eqref{eq:bou} as a differential inclusion and state its relaxation.
Let $\cS^{n\times n}$ be the set of all symmetric $n\times n$ matrices, $\cS_0^{n\times n}\subset \cS^{n\times n}$ the subset of matrices with vanishing trace and $\id\in \cS^{n\times n}$ be the identity matrix. We also write $\lamax(S),\lamin(S)$ for the maximal, minimal resp., eigenvalue of $S\in\cS^{n \times n}$, and the trace free part of $S$ is denoted by $S^\circ:=S-\frac{1}{n}\tr(S)\id$.

Consider on $\Omega\times (0,T)$ the linear system
\begin{align}\label{eq:subsolution_system_sec2}
\begin{split}
\partial_tv+\divv \sigma +\nabla p&=-\rho gA e_n,\\
\divv v&=0,\\
\partial_t\rho+\divv m&=0,
\end{split}
\end{align}
complemented with the boundary conditions
\begin{align}\label{eq:subsolution_boundary_data_sec2}
v\cdot\nu=0,\quad m\cdot\nu=0\quad \text{on }\partial\Omega\times(0,T),
\end{align}
for $z:=(\rho,v,m,\sigma,p)$ taking values in $Z:=\R\times\R^n\times\R^n\times\cS_0^{n\times n}\times\R$, and define
\begin{align}\label{eq:nonlinear_constraints_sec2}
K_{(x,t)}:=\set{z\in Z:\rho\in\{\pm 1\},~m=\rho v,~v\otimes v-\sigma=e(x,t)[\rho]\id}
\end{align}
for a given function $e:\Omega\times (0,T)\times \R\rightarrow\R$, $(x,t,r)\mapsto e(x,t)[r]$, which is affine linear in $r$. A brief discussion on possible choices of $e$ and some general constraints can be found in Section \ref{sec:choices_concerning_the_energy} below.

Now if $z:\Omega\times (0,T)\rightarrow Z$ is a weak solution of \eqref{eq:subsolution_system_sec2},  \eqref{eq:subsolution_boundary_data_sec2} to some initial data $(\rho_0,v_0)$ as in \eqref{eq:general_initial_data}, see Definition  \ref{def:subsolEuler} below for the precise definition, and if for almost every $(x,t)\in\Omega\times(0,T)$ there holds $z(x,t)\in K_{(x,t)}$, then $(\rho,v)$ defines a solution to the original equation \eqref{eq:bou} in the sense of Definition \ref{def:weaksols} for the same initial data and with energy density function given by
\[
\cE(x,t)=\frac{n}{2}e(x,t)[\rho(x,t)]+\rho(x,t)gAx_n.
\]
Conversely, if $(\rho,v)$ with associated pressure $p$ is a weak solution in the sense of Definition \ref{def:weaksols}, then 
$z=\left(\rho,v,\rho v,(v\otimes v)^\circ,p+\frac{1}{n}\abs{v}^2\right)$ is a weak solution of \eqref{eq:subsolution_system_sec2}, \eqref{eq:subsolution_boundary_data_sec2} and $z$ pointwise a.e. takes values in the set $K_{(x,t)}$ defined with respect to the function $e(x,t)[r]=\frac{1}{n}\abs{v(x,t)}^2$.

For the relaxation of \eqref{eq:subsolution_system_sec2}, \eqref{eq:nonlinear_constraints_sec2} let $Z_0:=\set{z\in Z:\rho\in(-1,1)}$, as well as $T_+,T_-,Q:Z_0\rightarrow \R$, $M:Z_0\rightarrow \cS^{n \times n}$,
\begin{gather}
\begin{gathered}\label{eq:definition_of_functions_for_hull}
M(z)=\frac{v\otimes v-\rho(m\otimes v+v\otimes m)+m\otimes m}{1-\rho^2}-\sigma,\\
Q(z)=\lamax(M(z)),\quad T_{\pm}(z)=\frac{\abs{m\pm v}^2}{n(\rho\pm 1)^2},
\end{gathered}
\end{gather}
and define for $(x,t)\in\Omega\times(0,T)$ the open set
\begin{equation}\label{eq:definition_of_U_sec2}
U_{(x,t)}:=\set{z\in Z:\rho\in(-1,1),~T_\pm(z)<e(x,t)[\pm 1],~Q(z)<e(x,t)[\rho]}.
\end{equation}
In the course of the article we will show that $U_{(x,t)}$ is the interior of the convex hull of $K_{(x,t)}$. That in particular means that if $(\rho_k,v_k)_{k\in\N}$ is a sequence of weak solutions with $v_k\in L^\infty(\Omega\times(0,T);\R^n)$ and such that the following convergences hold true $(\rho_k,v_k,\rho_kv_k,(v_k\otimes v_k)^\circ)\overset{*}{\rightharpoonup}(\rho,v,m,\sigma)$ in $L^\infty(\Omega\times(0,T);\R\times\R^n\times\R^n\times \cS_0^{n\times n})$ and $\frac{1}{n}\abs{v_k}^2\rightarrow e$ in $L^\infty(\Omega\times (0,T))$, then there exists a pressure $p$, such that $(\rho,v,m,\sigma,p)$ is a weak solution of \eqref{eq:subsolution_system_sec2}, while pointwise a.e. taking values in $\overline{U}_{(x,t)}$, where $U_{(x,t)}$ is defined with respect $e$.

With the help of the linear system \eqref{eq:subsolution_system_sec2} and the sets \eqref{eq:definition_of_U_sec2} we are ready to formulate the notion of subsolutions to \eqref{eq:bou}, as well as our general convex integration result. Doing this the following projection turns out to be convenient:
for $z=(\rho,v,m,\sigma,p)\in Z$ let
\begin{align}\label{eq:projection}
\pi(z):=(\rho,v,m,\sigma)\in \R\times\R^n\times\R^n\times\mathcal S_0^{n \times n}.
\end{align}
\begin{definition}[Subsolutions]\label{def:subsolEuler}
Let $e:\Omega\times(0,T)\times[-1,1]\rightarrow \R$ be bounded and affine linear in the last component.
We say that $z=(\rho,v,m,\sigma,p):\Omega\times(0,T)\to Z$ is a subsolution of \eqref{eq:bou} associated with $e$ and initial data $(\rho_0,v_0)$ as in \eqref{eq:general_initial_data} if and only if $\pi(z)\in L^\infty(\Omega\times(0,T);\pi(Z))$, $p$ is a distribution, 
$z$ solves \eqref{eq:subsolution_system_sec2}, \eqref{eq:subsolution_boundary_data_sec2} in the sense that $v$ is weakly divergence-free (as in Definition \ref{def:weaksols}), 
\begin{align*}
\int_0^T\int_{\Omega} \left[v \cdot \partial_t \Phi + \langle \sigma ,\nabla\Phi\rangle - gA \rho \Phi_n \right] \ dx \ dt +\int_{\Omega}v_0 (x)\cdot \Phi(x,0)\ dx=0,\\
\int_0^T\int_{\Omega} \left[\rho \partial_t \Psi + m\cdot\nabla\Psi \right] \ dx \ dt +\int_{\Omega} \rho_0 (x) \Psi(x,0)\ dx=0,
\end{align*}
for any test functions $\Phi\in C^\infty_c(\Omega\times[0,T);\mathbb{R}^2)$, $\divv \Phi =0$, $\Psi\in C^\infty_c(\overline{\Omega}\times[0,T))$,   
and if there exists an open set $\mathscr{U}\subset \Omega\times(0,T)$, such that the two  restricted maps $\mathscr{U}\ni(x,t)\mapsto \pi(z(x,t))\in\pi(Z)$ and $\mathscr{U}\times\R\ni (x,t,r)\mapsto e(x,t)[r]\in\R$ are continuous, and if there holds 
$
z(x,t)\in U_{(x,t)}$ for all $(x,t)\in\mathscr{U}$, as well as $z(x,t)\in K_{(x,t)}$  for a.e. $(x,t)\in\Omega\times(0,T)\setminus\mathscr{U}$. The open set $\mathscr{U}$ is called the mixing zone of $z$, and in analogy to solutions we call the subsolution admissible provided 
\begin{align}\label{eq:subsolenergy}
\mathcal E_{sub}(x,t):=\frac{n}{2}e(x,t)[\rho(x,t)]+\rho(x,t) gAx_n\end{align}
satisfies
\begin{align}\label{eq:weakadm}
\int_\Omega \mathcal E_{sub}(x,t)\, dx\leq \int_\Omega \frac{1}{2}\abs{v_0(x)}^2+\rho_0(x)gAx_n\,dx \text{ for a.e. }t\in(0,T).
\end{align}
\end{definition}

Before formulating our convex integration theorem we like to point out the following observation, which follows from Lemma 8 in \cite{DeL-Sz-Adm}. 
\begin{remark}\label{rem:regulartiy_of_rho}
Without loss of generality the $\rho$-component of any subsolution or solution is contained in $\cC^0([0,T];L^2_w(\Omega))$. That is for any $w\in L^2(\Omega)$ the function $[0,T]\ni t\mapsto \int_\Omega \rho(x,t)w(x)\:dx\in\R$ is continuous. 
\end{remark}
More precisely, \cite[Lemma 8]{DeL-Sz-Adm} gives $\rho\in\cC^0((0,T);L^2_w(\Omega))$, but looking into the proof one sees that the  functions in \cite[equation (90)]{DeL-Sz-Adm} can be uniquely extended to $\cC^0([0,T])$.

Observe also that outside the mixing zone $\mathscr{U}$ the components $(\rho,v)$ of a subsolution $z$ already solve the Euler-Boussinesq equation \eqref{eq:bou}.

\begin{theorem}\label{thm:main2} Let $z=(\rho,v,m,\sigma,p)$ be a subsolution associated with $e$ and initial data $(\rho_0,v_0)$ satisfying \eqref{eq:general_initial_data}, where  $e:\Omega\times(0,T)\times[-1,1]\rightarrow\R$ is given by
\[
e(x,t)[r]=e_0(x,t)+re_1(x,t)
\]
with $e_0\in L^\infty(\Omega\times(0,T))$, $e_1\in L^\infty(\Omega\times (0,T))\cap \cC^0([0,T];L^2(\Omega))$.  Then for any error function $\delta:[0,T]\rightarrow \R$, $\delta(0)=0$, $\delta(t)>0$, $t>0$ there exist infinitely many weak solutions $(\rho_{sol},v_{sol})$ of \eqref{eq:bou}, \eqref{eq:boundary_condition} with initial data $(\rho_0,v_0)$ having the properties
\begin{enumerate}[a)]
\item $(\rho_{sol},v_{sol})=(\rho,v)$ a.e. on $\Omega\times(0,T)\setminus\mathscr{U}$, 
\item the local energy density defined in \eqref{eq:local_energy_density} for a.e. $(x,t)\in\Omega\times (0,T)$ is given by 
\[
\cE_{sol}(x,t)=\frac{n}{2}e(x,t)[\rho_{sol}(x,t)]+\rho_{sol}(x,t)gAx_n,
\]
\item for any $t\in[0,T]$ there holds 
\[
\abs{\int_\Omega \left(\frac{n}{2}e_1(x,t)+gAx_n\right)(\rho(x,t)-\rho_{sol}(x,t)) \:dx}<\delta(t),
\]
\item for any $t\in (0,T)$ and any open ball $B\subset\Omega$ with $B\times \{t\}\subset\mathscr{U}$ there holds
\[
\int_B (1-\rho_{sol}(x,t))\:dx\int_B (1+\rho_{sol}(x,t))\:dx> 0.
\]
\end{enumerate}
Moreover, among these solutions one can find a sequence $(\rho_k,v_k)$, $k\in \N$, such that $\rho_k\rightarrow \rho$ in $\cC^0([0,T];L^2_w(\Omega))$ and $v_k\rightharpoonup v$ in $L^2(\Omega\times (0,T))$.
\end{theorem}

\begin{remark}[Admissibility]\label{rem:admissibility}
Observe that by Remark \ref{rem:regulartiy_of_rho} and the assumption on $e_1$ the integral on the left-hand side in Thm. \ref{thm:main2} c) defines a continuous function on $[0,T]$. Moreover, for a.e. $t\in(0,T)$ the energy difference between the subsolution and the solutions is precisely given by this term, i.e.
\[
\int_\Omega\cE_{sub}(x,t)-\cE_{sol}(x,t)\:dx=\int_\Omega\left(\frac{n}{2}e_1(x,t)+gAx_n\right)(\rho(x,t)-\rho_{sol}(x,t))\:dx
\]
for a.e. $t\in(0,T)$. In particular, if the subsolution is admissible with strict inequality in \eqref{eq:weakadm} for a.e. $t\in (0,T)$, then by a suitable choice of error function $\delta(t)$ one sees that property c) implies the admissibility of the induced solutions $(\rho_{sol},v_{sol})$ in the sense of Definition \ref{def:admissibility_of_weak_solutions}. 
\end{remark}

\begin{remark}[Mixing]
The convergence $\rho_k\rightarrow \rho$ in $\cC^0([0,T];L^2_w(\Omega))$ means that for any $w\in L^2(\Omega)$ there holds 
\[
\sup_{t\in[0,T]}\abs{\int_{\Omega}(\rho_k(x,t)-\rho(x,t))w(x)\:dx}\rightarrow 0.
\]
In that sense at every $t\in[0,T]$ the subsolution density $\rho(\cdot,t)$ can be seen as a coarse grained or averaged density of the induced solutions $\rho_{sol}(\cdot,t)$, whose turbulent nature is illustrated by means of the mixing at every time slice property d).
\end{remark}

The proof of Theorem \ref{thm:main2} will be carried out in Section \ref{sec:tartar} and is based on the convex integration methods introduced by De Lellis and Sz\'ekelyhidi in \cite{DeL-Sz-Annals,DeL-Sz-Adm} and its refinements in \cite{Castro-Faraco-Mengual,Crippa}. In particular looking at \cite{Castro-Faraco-Mengual} one could in addition also add the ``linearly degraded macroscopic behaviour'' to the list of properties of the solutions in Theorem \ref{thm:main2}. Moreover, if one is interested in the notion of admissibility at every time, by which we mean that the inequality in Definition \ref{def:admissibility_of_weak_solutions} holds for all $t\in[0,T]$ instead of a.e. $t\in(0,T)$, one can use the convex integration strategy from \cite{Castro-Faraco-Mengual,DeL-Sz-Adm} based on a ``shifted grid'', which is not used here.

\subsection{Choices for \texorpdfstring{$e(x,t)[r]$}{the energy}}\label{sec:choices_concerning_the_energy}
In order to have inside the mixing zone $\mathscr{U}$ of a subsolution a non-empty interior of the convex hull $U_{(x,t)}$ we need 
\begin{equation}\label{eq:condition_on_e}
e(x,t)[\pm 1]>0,\quad\text{for all }(x,t)\in\mathscr{U}.
\end{equation}
In general $e(x,t)[r]$ has to be non-negative a.e., because this expression coincides up to a positive factor with the kinetic energy of the solutions.

Besides the above conditions one can a priori use for $e(x,t)[r]$ any function of the type
\[
e(x,t)=e_0(x,t)+e_1(x,t)r
\]
with $e_0,e_1$ continuous on $\mathscr{U}$, but in fact we will only consider such $e$ with 
\begin{equation}\label{eq:choice_of_e_1}
e_1(x,t)=-\varepsilon gAx_n,\quad \varepsilon\in\left[0,\frac{2}{n}\right].
\end{equation}
With this choice the solutions obtained by Theorem \ref{thm:main2} will have a kinetic energy a.e. given by
\[
\frac{1}{2}\abs{v_{sol}(x,t)}^2=\frac{n}{2}e_0(x,t)-\frac{n}{2}\varepsilon gA x_n\rho_{sol}(x,t).
\] 
This means that besides the continuous part $\frac{n}{2}e_0(x,t)$, which can be seen as a non turbulent or averaged part, the kinetic energy density of the solutions absorbs a certain fraction, given by $\frac{n}{2}\varepsilon\in[0,1]$, of the turbulent oscillations in the potential energy density $gAx_n\rho_{sol}(x,t)$. 

A priori also $\varepsilon$ can be a function depending on $(x,t)$, but we will mostly stick to constant $\varepsilon$, except for Section \ref{sec:beyond_small_times}.

\subsection{Subsolutions}\label{sec:statement_subsol}

Our second main result addresses the construction and selection of subsolutions associated with the initial data $\rho_0=\text{sgn}(x_n)$, $v_0\equiv 0$. We consider the problem on an $n$-dimensional box $\Omega=(0,1)^{n-1}\times(-L,L)$, $L>0$, $n\geq 2$ and focus on self-similar subsolutions. For the precise definition let $\cF$ denote the set of all $f\in \cC^1([-1,1])$ satisfying
\begin{gather}\label{eq:definition_of_set_F}
f(\pm 1)=\pm 1,~ f'(\pm 1)> 0,~ f(y)\in(-1,1),~ f(-y)=-f(y),~y\in(-1,1)
\end{gather}
and let $\cA$ denote the set of all $a\in \cC^2([0,T))$ with 
\begin{equation}\label{eq:definition_of_set_A}
a(0)=0,~ a(t)>0,~t\in(0,T).
\end{equation}
In Section \ref{sec:1Dsubsols} we will prove the following lemma.

\begin{lemma}\label{lem:1D_subsolutions}
Any triple $(f,a,\varepsilon)\in\cF\times\cA\times\left[0,\frac{2}{n}\right]$ gives rise to a continuous, piecewise $\cC^1$ subsolution $z$ with 
\[
\rho(x,t)=\begin{cases}
1,&x_n\geq a(t),\\
f\left(\frac{x_n}{a(t)}\right),&x_n\in(-a(t),a(t)),\\
-1,&x_n\leq -a(t),
\end{cases}
\]
$v\equiv 0$, $m_i\equiv 0$, $1\leq i\leq n-1$, as long as $a(t)\leq L$ and with $e$ having the form $e(x,t)[r]=e_0(x,t)-\varepsilon gAx_n r$.
\end{lemma}
We refer to these subsolutions as self-similar subsolutions. Considering solutions with $v\equiv 0$, $m_i\equiv 0$, $1\leq i\leq n-1$ and independent of $x_1,\ldots,x_{n-1}$ reflects the interpretation of the subsolution as an $(x_1,\ldots,x_{n-1})$-averaged solution. Moreover, we will see that the symmetry condition on $f$ is needed for the existence of self-similar subsolutions for the Boussinesq system. In contrast the subsolution constructed in \cite{GKSz} for the Euler system without Boussinesq approximation is also self-similar, but the profile $f$ is not symmetric.

Note that the associated mixing zone is given by $\mathscr{U}_a:=\set{(x,t):\abs{x_n}<a(t)}$. Note also that at this point the subsolutions are not necessarily admissible. 

In order to investigate the admissibility let $z=z_{f,a,\varepsilon}$ be a self-similar subsolution and define the function $\tilde{e}_{f,a,\varepsilon}:\Omega\times(0,T)\rightarrow\R$,
\begin{align*}
\tilde{e}_{f,a,\varepsilon}(x,t):=\inf&\left\{e_0(x,t):e_0\in L^\infty(\Omega\times(0,T))\cap\cC^0(\mathscr{U}_a),\right.\\
&\hspace{30pt}\left.z_{f,a,\varepsilon}\text{ is a subsolution w.r.t. }e(x,t)[r]=e_0(x,t)-\varepsilon gAx_n r\right\}.
\end{align*}
Hence by this definition, Theorem \ref{thm:main2} c) and Remark \ref{rem:admissibility} the subsolution $z_{f,a,\varepsilon}$ induces mixing solutions whose total energy $\int_\Omega\cE_{sol}(x,t)\:dx$ for a.e. $t\in(0,T)$ is arbitrarily close to
\begin{align}\label{eq:subste}
E_{f,a,\varepsilon}(t):=\int_\Omega\frac{n}{2}\tilde{e}_{f,a,\varepsilon}(x,t)+\left(1-\frac{n}{2}\varepsilon\right)gAx_n \rho_{f,a,\varepsilon}(x,t)\:dx.
\end{align}

Note that if $z_{f,a,\varepsilon}$ is admissible, then $E_{f,a,\varepsilon}(t)\leq E(0)$ for a.e. $t\in (0,T)$, where $E(0)=gAL^2$ is the initial energy associated with \eqref{eq:initial_data}. In order to evaluate the initial loss of energy define for $k=0,\ldots,4$ the functionals
\begin{align}\label{eq:Jk}
J_k(f,a,\varepsilon):=\lim_{t\rightarrow+\infty}\frac{E_{f,a,\varepsilon}(t)-E(0)}{t^k},
\end{align}
whenever the limits exist. We have the following small time selection of a self-similar subsolution.
\begin{theorem}\label{thm:selection_by_initial_dissipation}
For any $f\in\cF$, $a\in \cA$, $\varepsilon\in \left[0,\frac{2}{n}\right]$, such that $z_{f,a,\varepsilon}$ is admissible there holds $\dot{a}(0)=0$ and $J_k(f,a,\varepsilon)=0$, $k=0,1,2,3$. Moreover, among all admissible self-similar subsolutions the maximal initial dissipation rate
\[
\inf\set{J_4(f,a,\varepsilon):(f,a,\varepsilon)\in\cF\times\cA\times\left[0,\frac{2}{n}\right],~J_k(f,a,\varepsilon)=0\text{ for }k=0,1,2,3}
\]
is achieved for $f(y)=y$, $a(t)=\frac{1}{3}gAt^2+o(t^2)$, $\varepsilon=\frac{2}{3n}$. Up to the $o(t^2)$, the minimizer is unique.
\end{theorem}
We will see that for $f(y)=y$, $a(t)=\frac{1}{3}gAt^2$ and $\varepsilon=\frac{2}{3n}$ there holds 
\begin{align*}
E_{f,a,\varepsilon}(t)-E(0)=-\frac{1}{81}g^3A^3t^4
\end{align*}
as long as $a(t)\leq L$, i.e. for all $t\in\left[0,\sqrt{\frac{3L}{gA}}\right]$.

Next we will formulate the two statements concerning the extension of the subsolution to all times.
We like to emphasize that for the extensions we no longer use a selection criterion, instead the constructions contain several choices and for now are only done to illustrate possible options for the long-time behaviour.

\begin{proposition}\label{prop:mix}
The minimizing subsolution from Theorem \ref{thm:selection_by_initial_dissipation} with $o(t^2)=0$ can be extended in an admissible manner to $\Omega\times(0,+\infty)$ such that it converges to the fully mixed, isotropic state $z\equiv 0$ as $t\to+\infty$ and such that also the associated kinetic energy $\frac{n}{2}e(x,t)[\rho(x,t)]$ converges to $0$.
\end{proposition}

\begin{proposition}\label{prop:demix}
There exists 
$T_{end}\in\left(\sqrt{\frac{3L}{gA}},+\infty\right)$
such that the minimizing subsolution from Theorem \ref{thm:selection_by_initial_dissipation} with $o(t^2)=0$ can be extended in an admissible manner to $\Omega\times(0,T_{end})$, and at $T_{end}$ it reaches the stable configuration $\rho=-\rho_0$, $(v,m,\sigma)\equiv 0$,  $p=const.$, $e(\cdot,T_{end})[\cdot]\equiv 0$.
\end{proposition}

In fact, both subsolutions are not only admissible, but satisfy the strong energy inequality, which means that the total energy $\int_\Omega\cE_{sub}(x,t)\:dx$ is monotone decreasing w.r.t. time. 

Moreover, in the first case the subsolution satisfies $z(x,t)\in U_{(x,t)}$ for every $x\in\Omega$ and $t>\sqrt{\frac{3L}{gA}}$ while the closure of the hull $\overline{U}_{(x,t)}$ collapses as $t\rightarrow +\infty$ to the set $[-1,1]\times\{0\}\times\{0\}\times\{0\}\times\R\subset Z$
due to the decay of kinetic energy. Thus technically the mixing zone is unbounded here. In the second case we have that $z(x,T_{end})$ actually is a solution, i.e. $z(x,T_{end})\in K_{(x,T_{end})}$ for a.e. $x\in \Omega$. Clearly we can extend this subsolution to all times by  $z(\cdot,t)=z(\cdot,T_{end})$ for all $t>T_{end}$.

\section{Convex integration via the Tartar framework}\label{sec:tartar}

To prove our main result, we will use a version of the Tartar framework, originally introduced in the context of compensated compactness \cite{Tartar}, for differential inclusions when the set of nonlinear constraints is not constant (c.f. e.g. \cite{Castro-Faraco-Mengual,Crippa,DeL-Sz-Adm}).

The general strategy of convex integration in the Tartar framework relies on the idea that if one can find a weak solution $\tilde z$ of \eqref{eq:subsolution_system_sec2} which instead of taking values in $K_{(x,t)}$ satisfies $\tilde{z}(x,t)\in\text{int}\left(K_{(x,t)}^{co}\right)$, then one may deduce the existence of (infinitely many) solutions $z$ of \eqref{eq:subsolution_system_sec2}, which are near $\tilde{z}$ in the weak sense while satisfying $z(x,t)\in K_{(x,t)}$ a.e., by adding some specially constructed perturbations to $\tilde{z}$.
The perturbations rely on localized plane waves as basic building blocks.

\subsection{Localized plane waves}\label{sec:waves}
For $\bar{z}\in Z$ we define
\[
M_\Lambda(\bar{z}):=\begin{pmatrix}
\bar{\sigma}+\bar{p}\id & \bar{v}\\
\bar{v}^T & 0\\
\bar{m}^T & \bar{\rho}
\end{pmatrix}\in\R^{(n+2)\times(n+1)},
\]
such that the wave cone associated with \eqref{eq:subsolution_system_sec2} can be written as
\begin{equation}\label{eq:wave_cone}
\Lambda:=\set{\bar{z}\in Z:\ker M_\Lambda(\bar{z}) \neq \{0\},\quad (\bar\rho,\bar v)\neq 0}.
\end{equation}
Note that for $\bar{z}\in\Lambda$ there exists $\eta=(\xi,c)\in\R^{n+1}\setminus\{0\}$ such that every function $z(x,t)=\bar{z}h((x,t)\cdot\eta)$, $h\in\cC^1(\R)$ is a solution of \eqref{eq:subsolution_system_sec2}. This allows us to construct solutions which oscillate in the direction $\bar z$.
Note that
the condition $(\bar\rho,\bar v)\neq 0$ allows us to exclude the degenerate case when $\xi=0$, which would correspond to having only oscillations in time.

Let us define a restricted wave cone which also eliminates oscillations only in space, i.e.
\begin{equation}\label{eq:wave_cone_restricted}
\Lambda':=\set{\bar{z}\in \Lambda:\ \ker M_\Lambda(\bar{z})\cap \R^n\times (\R\setminus\{0\})\neq \emptyset }.
\end{equation}

In Lemma \ref{lem:locpw} below we construct localized plane wave-like solutions for \eqref{eq:subsolution_system_sec2} associated with $\bar{z}\in\Lambda'$. In order to see that it is enough to consider $\Lambda'$ instead of $\Lambda$ we first show the following density lemma.

\begin{lemma}\label{lem:density_of_lambda_prime}
The restricted cone $\Lambda'$ is dense in $\Lambda$.
\end{lemma}
\begin{proof}
Let $\bar{z}\in\Lambda\setminus\Lambda'$. It follows that there exists $\xi\in S^{n-1}$ such that $\bar{v}\cdot \xi=0$, and we also have $\bar m \cdot \xi=0$, $(\bar\sigma+\bar p \id)\xi=0$.

We define the following sequence. For $N\geq 1$ let
\begin{multline*}
\bar\rho_N:=\bar\rho+\frac{1}{N},\quad \bar v_N:=\bar{v},\quad \bar m_N:=\bar m+\frac{1}{N^2}\xi,\\ \bar\sigma_N+\bar p_N \id :=\bar\sigma+\bar p \id+\frac{1}{N^2\bar\rho+N}\left(\xi\otimes\bar v+\bar v\otimes\xi \right).
\end{multline*}
Here and in forthcoming formulas the definition of $\bar{\sigma}_N$ and $\bar{p}_N$ is understood in the sense that the symmetric matrix on the right hand side is split into its trace free part and its trace.

It is easy to check that $\left(\xi,-\frac{1}{N^2\bar\rho+N} \right)\in \ker M_\Lambda(\bar{z}_N)$, therefore $\bar{z}_N\in\Lambda'$ for $N\geq 1$. Furthermore, clearly $\bar z_N\to\bar z$ as $N\to+\infty$. 
This concludes the proof.
\end{proof}
Recall the definition of the projection $\pi:Z\rightarrow \R\times\R^n\times\R^n\times \cS_0^{n\times n}$ from \eqref{eq:projection}. We write $d$ for the euclidian distance function.

\begin{lemma}\label{lem:locpw}
There exists $C>0$ such that for any $\bar{z}\in\Lambda'$, there exists a sequence 
$z_N\in C_c^\infty(B_1(0);Z)$, where $B_1(0)\subset\R^n\times\R$, solving the linear system \eqref{eq:subsolution_system_sec2} and satisfying 
\begin{itemize}
\item[(i)] $d(z_N,[-\bar{z},\bar{z}])\to 0$ uniformly,
\item[(ii)] $z_N\rightharpoonup 0$ in $L^2(B_1(0);Z)$,
\item[(iii)] $\int_{B_1(0)} |\pi(z_N)|^2\, d(x,t)\geq C|\pi(\bar{z})|^2.$
\end{itemize}
\end{lemma}
\begin{proof}
We will construct the desired sequence of solutions as a sum of two sequences $z_N=\hat{z}_N+\tilde{z}_N$, where $\hat{z}_N$ will be a localized plane wave for the usual Euler equations determining up to a small deviation $v_N$, $\sigma_N$ and $p_N$, while $\tilde{z}_N$ will take care of $\rho_N$ and $m_N$.

\textbf{Step 1. Euler-type plane waves.}

We treat two cases.
First, suppose that
 $\bar{z}\in \Lambda'$ with $\bar v\neq 0$. It follows from \cite{DeL-Sz-Annals,DeL-Sz-Adm} that there exists a sequence $(\hat{v}_N,\hat{\sigma}_N,\hat{p}_N)\subset \cC_c^\infty(B_1(0);\R^n\times\cS_0^{n\times n}\times\R)$ satisfying 
\begin{align*}
\partial_t\hat{v}_N+\divv\hat{\sigma}_N+\nabla \hat{p}_N=0,\quad 
\divv \hat{v}_N=0,
\end{align*}
and such that the distance between $(\hat{v}_N(x,t),\hat{\sigma}_N(x,t),\hat{p}_N(x,t))$ and the line segment $[-(\bar{v},\bar{\sigma},\bar{p}),(\bar{v},\bar{\sigma},\bar{p})]$ converges to $0$ uniformly in $(x,t)$, $(\hat{v}_N,\hat{\sigma}_N,\hat{p}_N)\rightharpoonup 0$ in $L^2$ and $\int_{B_1(0)}\abs{(\hat{v}_N,\hat{\sigma}_N,\hat{p}_N)}^2\:d(x,t)\geq \hat{C} \abs{(\bar{v},\bar{\sigma},\bar{p})}^2$ for a constant $\hat{C}>0$ independent of $\bar{z}$. We then define the whole vector $\hat{z}_N$ by setting $\hat{\rho}_N=0$ and $\hat{m}_N=0$. Clearly $\hat{z}_N$ satisfies \eqref{eq:subsolution_system_sec2}.

In the second case, if $\bar{z}\in \Lambda'$ such that $\bar v= 0$, one can not apply the construction from \cite{DeL-Sz-Annals,DeL-Sz-Adm}, however one may construct a different suitable potential in the following way.
We know that by the definition of $\Lambda'$ there exists $\eta=(\xi,c)\in\R^n\times\R$ with $c\neq 0$ and $M_\Lambda(\bar{z})\eta=0$. In particular $\bar{m}\cdot\xi+\bar{\rho}c=0$. As already discussed before there necessarily holds $\xi\neq 0$, because otherwise $(\bar\rho,\bar{v})=0$, which is ruled out by the definition of $\Lambda'$. However, since $\bar v=0$, we also obtain $(\bar\sigma+\bar p\id)\xi=0.$

If $n=2$, then this implies that $\bar\sigma+\bar p\id=k_1\xi^\perp\otimes \xi^\perp$ for some $k_1\in\R$. Here $\xi^\perp:=(-\xi_2,\xi_1)$. Furthermore, for any $\Psi\in\cC^\infty(\R^2\times\R)$, setting
\begin{align}
\begin{gathered}\label{eq:pott}
\hat\sigma+\hat p\id:=(\nabla^\perp)^2\Psi=\begin{pmatrix}
\partial_{x_2}^2\Psi & -\partial_{x_1}\partial_{x_2}\Psi\\
-\partial_{x_1}\partial_{x_2}\Psi &\partial_{x_1}^2\Psi
\end{pmatrix},\\
 \hat v\equiv 0,\ \hat\rho\equiv0,\ \hat m\equiv 0,
\end{gathered}
\end{align}
 yields a solution of \eqref{eq:subsolution_system_sec2}. 
In particular, setting
$$\Psi_N(x,t):=k_1\frac{1}{N^2}\sin(N(x\cdot\xi+tc))\chi_\varepsilon(x,t),$$
where $\chi_\varepsilon\in\cC^\infty_c(B_1(0))$ satisfies $\abs{\chi_\varepsilon}\leq 1$ on $B_1(0)$, $\chi_\varepsilon=1$ on $B_{1-\varepsilon}(0)$,
one obtains that the function $\hat z_N$ associated via \eqref{eq:pott} satisfies
\[
\hat\sigma_N(x,t)+\hat p_N(x,t)\id=-(\bar\sigma+\bar p\id)\sin(N(x\cdot\xi+tc))\chi_\varepsilon(x,t)+O(1/N)
\]
uniformly in $(x,t)$ as $N\rightarrow+\infty$.
The remaining properties in analogy to the first case then follow in the usual way, cf. in particular Lemma 7 in \cite{DeL-Sz-Adm}.

If $n=3$, it follows that $(0,\xi)$ is an eigenpair of $\bar\sigma+\bar p\id$, hence by a spectral decomposition one obtains that $\bar\sigma+\bar p\id=\lambda_1\nu_1\otimes\nu_1+\lambda_2\nu_2\otimes\nu_2$, for some $\lambda_{1,2}\in\R$ and $\nu_{1,2}\perp\xi$.
Assume without loss of generality that the second component of $\xi$ is not vanishing, such that one may write $\nu_{1,2}$ as linear combinations of $\xi^\perp_1:=(-\xi_2,\xi_1,0)^T$ and $\xi^\perp_2:=(0,-\xi_3,\xi_2)^T$. Otherwise, i.e. if $\xi_2=0$, one can use the corresponding pair of linear independent orthogonal vectors associated with $\xi_1\neq 0$ or $\xi_3\neq 0$.
These linear combinations allow us to deduce that 
there exist some $k_1,k_2,k_3\in\R$ such that 
\begin{align}\label{eq:goodluck}\bar\sigma+\bar p\id=k_1\xi_1^\perp\otimes \xi_1^\perp+k_2\xi_2^\perp\otimes \xi_2^\perp+k_3(\xi_1^\perp\otimes \xi_2^\perp+\xi_2^\perp\otimes \xi_1^\perp).
\end{align}
Observe that, for any $\Phi\in\cC^\infty(\R^3\times\R;\R^3)$  setting 
\begin{align}\label{eq:pottt}
\begin{split}
\hat\sigma+\hat p\id:=&\begin{pmatrix}
\partial_2^2\Phi_1 & -\partial_1\partial_2\Phi_1 & 0\\
-\partial_1\partial_2\Phi_1 &  \partial_1^2\Phi_1 & 0\\
0 & 0 & 0
\end{pmatrix}
+\begin{pmatrix}
0 & 0 & 0\\
0&\partial_3^2\Phi_2 & -\partial_2\partial_3\Phi_2 \\
0&-\partial_2\partial_3\Phi_2 &  \partial_2^2\Phi_2 
\end{pmatrix}\\
&\hspace{45pt}+\begin{pmatrix}
0 & \partial_2\partial_3\Phi_3 & -\partial_2^2\Phi_3\\
\partial_2\partial_3\Phi_3 & -2\partial_1\partial_3\Phi_3 & \partial_1\partial_2\Phi_3\\
-\partial_2^2\Phi_3 & \partial_1\partial_2\Phi_3 & 0
\end{pmatrix},\\ \hat v\equiv 0,\  \hat\rho&\equiv0,\ \hat m\equiv 0,
\end{split}
\end{align} 
yields a solution of \eqref{eq:subsolution_system_sec2}. We then choose
\begin{align*}
\Phi_{N}(x,t):=(k_1,k_2,k_3)\frac{1}{N^2}\sin(N(x\cdot\xi+tc))\chi_\varepsilon(x,t),
\end{align*}
to obtain by \eqref{eq:goodluck} that the function $\hat z_N$ associated via \eqref{eq:pottt} satisfies
$$\hat\sigma_N+\hat p_N\id=-(\bar\sigma+\bar p\id)\sin(N(x\cdot\xi+tc))\chi_\varepsilon(x,t)+O(1/N).$$
One then concludes as in the case $n=2$.

For higher dimensions, one may proceed analogously, the details are left to the reader. This concludes the first step of our construction.

\textbf{Step 2. The potential for $\bar\rho$ and $\bar m$.}

We will show that there exists a constant $\tilde{C}>0$ independent of $\bar{z}$, and a sequence $\tilde{z}_N\subset\cC^\infty_c(B_1(0);Z)$ of solutions of \eqref{eq:subsolution_system_sec2}, such that
\begin{itemize}
\item[a)] $(\tilde{v}_N,\tilde{\sigma}_N,\tilde{p}_N)\rightarrow 0$ uniformly,
\item[b)] $d\big((\tilde{\rho}_N,\tilde{m}_N),[-(\bar{\rho},\bar{m}),(\bar{\rho},\bar{m})]\big)\rightarrow 0$ uniformly,
\item[c)] $(\tilde{\rho}_N,\tilde{m}_N)\rightharpoonup 0$ in $L^2(B_1(0);\R\times\R^n)$,
\item[d)] $\int_{B_1(0)}\abs{(\tilde{\rho}_N,\tilde{m}_N)}^2\:d(x,t)\geq \tilde{C}\abs{(\bar{\rho},\bar{m})}^2$.
\end{itemize}
It is clear that $C:=\min\set{\tilde{C},\hat{C}}$ and $z_N:=\hat{z}_N+\tilde{z}_N$ then satisfies the properties stated in the lemma.

For the existence of $\tilde{z}_N$ observe that for any $\Psi\in\cC^\infty(\R^n\times\R;\cS^{n\times n})$ the function $\tilde{z}$ defined by
\begin{gather}
\begin{gathered}\label{eq:definition_of_second_plane_waves}
\tilde{\rho}:=\partial_t\divv\divv\Psi,\quad \tilde{v}:=gA\partial_{x_n}\divv \Psi-gAe_n\divv\divv\Psi,\\
\tilde{m}:=-\partial_t^2\divv\Psi,\quad 
\tilde{\sigma}+\tilde{p}\id:=-gA\partial_t\partial_{x_n}\Psi
\end{gathered}
\end{gather}
satisfies equation \eqref{eq:subsolution_system_sec2}. 

As before, there exists $\eta=(\xi,c)\in\R^n\times\R$ with $c\neq 0$ and $M_\Lambda(\bar{z})\eta=0$. In particular $\bar{m}\cdot\xi+\bar{\rho}c=0$ and necessarily $\xi\neq 0$. The functions $\tilde{z}_N$ are now defined as in \eqref{eq:definition_of_second_plane_waves} with $\Psi=\Psi_N$ given by
\begin{align*}
\Psi_N(x,t):=\frac{1}{N^3}\tilde{M}\sin\left(N(x\cdot\xi+tc)\right)\chi_\varepsilon(x,t),
\end{align*}
where $\chi_\varepsilon\in\cC^\infty_c(B_1(0))$ again is the usual cut-off function and 
\begin{align*}
\tilde{M}:=\frac{\bar{m}\cdot\xi}{c^2\abs{\xi}^4}\xi\otimes\xi-\frac{1}{c^2\abs{\xi}^2}(\bar{m}\otimes\xi+\xi\otimes\bar{m})\in\cS^{n\times n}.
\end{align*}
Clearly a) holds true, since the definition of $\tilde{v}_N$, $\tilde{\sigma}_N$ and $\tilde{p}_N$ in \eqref{eq:definition_of_second_plane_waves} involves only derivatives of order $2$. Moreover,
\begin{align*}
\tilde{m}_N(x,t)&=c^2\tilde{M}\xi \cos\left(N(x\cdot \xi+tc)\right)\chi_\varepsilon(x,t)+O(1/N),\\
\tilde{\rho}_N(x,t)&=c\xi^T\tilde{M}\xi \cos\left(N(x\cdot \xi+tc)\right)\chi_\varepsilon(x,t)+O(1/N)
\end{align*}
uniformly as $N\rightarrow +\infty$, 
and $c^2\tilde{M}\xi=-\bar{m}$, $c\xi^T\tilde{M}\xi=-\bar{m}\cdot\xi/c=\bar{\rho}$. This shows property b). Properties c) and d) follow again in the standard fashion.
\end{proof}
\begin{remark}\label{rem:slightly_stronger_convergence_of_plane_waves}
If one replaces in Lemma \ref{lem:locpw} the space-time ball $B_1(0)\subset \R^n\times \R$ by the cylinder $B_1(0)\times (-1,1)\subset \R^n\times\R$ and then chooses a cutoff function $\chi_\varepsilon\in\cC^\infty_0(B_1(0)\times (-1,1))$ of the form $\chi_\varepsilon(x,t)=\chi^a_\varepsilon(x)\chi_\varepsilon^b(t)$ with $\chi^a_\varepsilon\in\cC^\infty_0(B_1(0))$, $\chi_\varepsilon^b\in\cC^\infty_0(-1,1)$ one sees that the convergence in (ii) improves to $z_N\rightarrow 0$ in $\cC^0([-1,1];L^2_w(B_1(0)))$.
\end{remark}

\subsection{Perturbing along sufficiently long segments}\label{sec:segments}

In this subsection we prove that the wave cone $\Lambda$ is large with respect to $K_{(x,t)}$, in the sense that any two points in $K_{(x,t)}$ can be connected with a $\Lambda$-segment. Furthermore, this property automatically implies that any point in the interior of the convex hull of $K_{(x,t)}$ can be perturbed along sufficiently long $\Lambda$-segments. The set $K_{(x,t)}$ has been defined in \eqref{eq:nonlinear_constraints_sec2}.

For simplicity of notation, for the rest of the subsection we will fix a point $(x,t)\in\Omega\times[0,T)$ and write $K$ instead of $K_{(x,t)}$.

 \begin{lemma}\label{lem:big_fkn_cone}
 For any $z_1,z_2\in K$, $z_1\neq z_2$, $p_1=p_2$, we have $\bar{z}:=z_2-z_1\in\Lambda$.
 \end{lemma}
 \begin{proof}  
 If $\bar{\rho}=0$, then $\bar{v}\neq 0$, because otherwise $z_1=z_2$. Furthermore, $\bar{m}=\rho_1\bar{v}$, so all that needs to be checked is that there exists $\xi\in\bar{v}^\perp$, $\xi\neq 0$ such that $\bar{\sigma}\xi=c\bar{v}$, for some $c\in\R$. However, for any $\xi\in \bar{v}^\perp$ there holds
 \begin{align*}
 \bar{\sigma}\xi&=((v_2\otimes v_2)^\circ-(v_1\otimes v_1)^\circ)\xi=(v_2\otimes v_2-v_1\otimes v_1)\xi\\
 &=(v_2\otimes \bar{v}+\bar{v}\otimes v_1)\xi=(v_1\cdot\xi)\bar{v}.
 \end{align*}
 
  If $\bar{\rho}\neq 0$, then without loss of generality it is equal to $2$, i.e. $\rho_2=1$, $\rho_1=-1$, and we obtain as before from  $z_1,z_2\in K$ that
  $\bar{\sigma}\xi=(v_1\cdot\xi)\bar{v}$ for any $\xi\in\bar{v}^\perp$.
  So it remains to check that $\bar{m}\cdot\xi=2v_1\cdot\xi$ also holds. We have
  $$\bar{m}\cdot\xi=(v_2+v_1)\cdot\xi=(\bar{v}+2v_1)\cdot\xi=2(v_1\cdot\xi)$$
  and the proof is finished.
 \end{proof}
 
Without having any further information on the convex hull of $K$, we can prove the following geometric lemma solely based on this property.
 \begin{corollary}\label{cor:long_segments}
  For any $z\in\text{int} (K^{co})$ there exists $\bar{z}\in \Lambda$ such that
 $$[z-\bar z,z+\bar z]\subset \text{int} (K^{co})\text{ and }|\pi(\bar z)|\geq\frac{1}{2N}d(\pi(z),\pi(K)),$$
 where $N=\text{dim}(Z)$ and $d$ is the Euclidean distance on $\pi(Z)$.
 \end{corollary}
 The proof is the same as those of Lemma 6 from \cite{DeL-Sz-Adm}, respectively Lemma 4.9 from \cite{GKSz}, relying on Carath\'eodory's theorem and Lemma \ref{lem:big_fkn_cone} above, therefore we omit it.

\subsection{The convex hull}\label{sec:convex_hull}

We now explicitly compute the full $\Lambda$-convex hull associated with the differential inclusion \eqref{eq:subsolution_system_sec2}, \eqref{eq:nonlinear_constraints_sec2}, which in our case turns out to coincide with the usual convex hull.
The definition of the $\Lambda$-convex hull $(K')^\Lambda$ of $K'\subset Z$ can be recalled for example from \cite{Kirchheim}.

Let us again fix $(x,t)\in\Omega\times[0,T)$ and write $K$ instead of $K_{(x,t)}$, $U$ instead of $U_{(x,t)}$ and $e[r]$ instead of $e(x,t)[r]$ for $r\in\R$. Recall the definition of $U$ in \eqref{eq:definition_of_U_sec2} and of the functions $T_\pm$, $Q$ in \eqref{eq:definition_of_functions_for_hull}.

\begin{proposition}\label{prop:lambda_convex_hull}
There holds $K^\Lambda=K^{co}=\overline{U}$.
\end{proposition}

In Lemma \ref{lem:U_convex_closure_of_U} below we will see that the closure of $U$ splits into 
\[
\overline{U}=K_-'\cup \overline{U}_0\cup K_+',
\]
where
\begin{align*}
\overline{U}_0&:=\set{z\in Z:\rho\in(-1,1),~T_\pm(z)\leq e[\pm 1],~Q(z)\leq e[\rho]},\\
K_\pm'&:=\set{z\in Z:\rho=\pm 1,~m=\pm v,~\lamax(v\otimes v-\sigma)\leq e[\pm 1]}.
\end{align*}
Moreover, Lemma \ref{lem:boundedness_of_U_extreme_points} actually shows that $K_\pm'$ is the $\Lambda$-convex hull of the sets  $K_\pm:=K\cap\{z\in Z:\rho=\pm 1\}$.

The proof of Proposition \ref{prop:lambda_convex_hull} is organized as in the corresponding Section of \cite{GKSz} for the inhomogeneous Euler equation and relies on Lemma \ref{lem:U_convex_closure_of_U} and \ref{lem:boundedness_of_U_extreme_points}.

\begin{lemma}\label{lem:Q_convex} 
The function $Q$ is convex.
\end{lemma}
\begin{proof}
Since $Q$ is defined as the maximal eigenvalue of the matrix $M(z)$, there holds
\[
Q(z)=\sup_{\xi\in S^{n-1}}\xi^TM(z)\xi=\sup_{\xi\in S^{n-1}}\left(g_\xi(z)-\xi^T\sigma \xi\right),
\]
where for every fixed $\xi\in S^{n-1}$ the function $g_\xi:\set{z\in Z:\rho\in(-1,1)}\rightarrow \R$ is given by 
\begin{align*}
g_\xi(z)&=\xi^TM(z)\xi+\xi^T\sigma \xi=\frac{(v\cdot \xi)^2-2\rho (m\cdot \xi)(v\cdot\xi)+(m\cdot \xi)^2}{1-\rho^2}.
\end{align*}
We will show that every $g_\xi$ is convex, such that $Q$ is convex as a supremum of convex functions. 

In order to prove the convexity of $g_\xi$, $\xi\in S^{n-1}$ fixed, we write $v=x_1 \xi+v'$, $m=x_2 \xi+m'$ with $x_1,x_2\in\R$, $v',m'\in \xi^\perp$. Then it is enough to show that the function $g:(-1,1)\times\R^2\rightarrow \R$,
\begin{align*}
g(\rho,x)=\frac{x_1^2-2\rho x_1x_2+x_2^2}{1-\rho^2}
\end{align*}
is convex. We write $g(\rho,x)=x^TA(\rho)x$ with
\[
A(\rho):=\frac{1}{1-\rho^2}\begin{pmatrix}
1 & -\rho \\
-\rho & 1
\end{pmatrix}.
\] 
Let us fix $(\rho,x)\in(-1,1)\times\R^2$ and observe that $A(\rho)$ is positive definite. Thus the restricted function $g(\rho,\cdot)$ is strictly convex, or equivalently $D^2g(\rho,x)[0,y]^2\geq  0$ for all $y\in\R^2$.

It therefore remains to show that $D^2g(\rho,x)[1,y]^2\geq 0$ for all $y\in\R^2$. By the positive definiteness of $A(\rho)$ we obtain
\begin{align*}
D^2g(\rho,x)[1,y]^2&=x^TA''(\rho)x+4y^TA'(\rho)x +2y^TA(\rho)y\\
&=2\left(y+A(\rho)^{-1}A'(\rho)x\right)^TA(\rho)\left(y+A(\rho)^{-1}A'(\rho)x\right)\\
&\phantom{=asd}+x^TA''(\rho)x -2x^TA'(\rho)A(\rho)^{-1}A'(\rho)x\\
&\geq x^T\left(A''(\rho) -2A'(\rho)A(\rho)^{-1}A'(\rho)\right)x.
\end{align*}
It turns out that in fact $A''(\rho)=2A'(\rho)A(\rho)^{-1}A'(\rho)$, which shows the convexity of $g$. 
Indeed, differentiating
\[
(1-\rho^2)A(\rho)=\begin{pmatrix}
1 & -\rho \\
-\rho & 1
\end{pmatrix}
\]
on both sides yields 
\begin{align}
(1-\rho^2)A'(\rho)&=2\rho A(\rho)+C,\label{eq:derivativeA}\\
\begin{split}
(1-\rho^2)^2A''(\rho)&=2(1+3\rho^2)A(\rho)+4\rho C,
\end{split}
\end{align}
where
\[
C:=\begin{pmatrix}
0 & -1\\
-1 & 0
\end{pmatrix}.
\]
Moreover, a straightforward computation shows
\begin{align}\label{eq:CAinversC}
CA(\rho)^{-1}C=(1-\rho^2)A(\rho)-2\rho C.
\end{align}
Now \eqref{eq:derivativeA}--\eqref{eq:CAinversC} imply the desired identity $A''(\rho)=2A'(\rho)A^{-1}(\rho)A'(\rho)$.
\end{proof}
The following lemma implies the inclusion $K^\Lambda\subset K^{co}\subset \overline{U}$.
\begin{lemma}\label{lem:U_convex_closure_of_U} The set $U$ is convex and its closure $\overline{U}$ splits into $\overline{U}=K_-'\cup\overline{U}_0\cup K_+'$.
In particular $K\subset \overline{U}$.
\end{lemma}
\begin{proof}
In Lemma \ref{lem:Q_convex} we have already shown that $Q$ is a convex function. Using the basic triangle inequality one can directly check that $T_\pm(z)<e[\pm 1]$ define convex sets. Hence $U$ is convex.

For the stated identity concerning $\overline{U}$ first of all observe that $\overline{U}_0\subset\overline{U}$. Next we will show that $K_\pm'\subset\overline{U}$. Let $z_*\in K_+'$ for instance and take any $z'\in K$ with $\rho'=-1$, as well as a sequence $(\rho_j)_{j\in\N}\subset (-1,1)$ with $\rho_j\rightarrow 1$. The element
\begin{align*}
z_j=\frac{1-\rho_j}{2}z'+\frac{1+\rho_j}{2}z_*
\end{align*}
clearly converges to $z_*$ as $j\rightarrow +\infty$. Using $z_*\in K_+'$ and $z'\in K$, $\rho'=-1$ one sees that 
\begin{align*}
T_+(z_j)=\frac{1}{n}\abs{v_*}^2
=\frac{1}{n}\tr(v_*\otimes v_*-\sigma_*)
\leq \lamax(v_*\otimes v_*-\sigma_*)
\leq e[1],
\end{align*}
as well as $T_-(z_j)=e[-1]$. For the matrix $M(z_j)$ we compute 
\begin{align*}
M(z_j)&=\frac{1-\rho_j}{2}\big(v'\otimes v'-\sigma'\big)
+\frac{1+\rho_j}{2}\big(v_*\otimes v_*-\sigma_*\big)\\
&=\frac{1-\rho_j}{2}e[-1]\id+\frac{1+\rho_j}{2}\big(v_*\otimes v_*-\sigma_*\big).
\end{align*}
Hence 
\[
Q(z_j)=\lamax(M(z_j))\leq \frac{1-\rho_j}{2}e[-1]+\frac{1+\rho_j}{2}e[1]=e[\rho_j].
\]
This shows that every $z_j$ and therefore also the limit $z_*$ is contained in $\overline{U}$. The case $z_*\in K_-'$ works analoguosly. Thus $K_-'\cup \overline{U}_0\cup K_+'\subset\overline{U}$.

Let now $(z_j)_{j\in\N}\subset U$ be convergent to some  $z_*\in Z$. If $\rho_*\in(-1,1)$, then it is clear that $z_*\in\overline{U}_0\subset\overline{U}$. Consider the case $\rho_*=1$. Since on $U$ there holds
\begin{align}\label{eq:Ubdd1}
|m\pm v|<\sqrt{ne[\pm 1]}(1\pm\rho),
\end{align}
it follows that $m_*=v_*$. Recall that $e[\pm 1]\geq 0$ from \eqref{eq:condition_on_e}. Next we rewrite
\begin{align}\label{eq:Ubdd2}
M(z)=v\otimes v+(1-\rho^2)\frac{m-\rho v}{1-\rho^2}\otimes \frac{m-\rho v}{1-\rho^2}-\sigma,
\end{align}
and observe that
\begin{align}\label{eq:Ubdd3}|m-\rho v|\leq\frac{1-\rho}{2}|m+v|
+\frac{1+\rho}{2}|m-v|<\frac{1}{2}\max\set{\sqrt{ne[-1]},\sqrt{ne[+1]}}(1-\rho^2),
\end{align}
by \eqref{eq:Ubdd1}. Therefore
\[
\lim_{j\rightarrow +\infty} M(z_j)= v_*\otimes v_*-\sigma_*.
\]
Thus $\lamax(M(z_j))<e[\rho_j]$, $j\in \N$ and the continuity of the maximal eigenvalue function imply $z_*\in K_+'$. The same procedure again works for the other case $\rho_*=-1$, such that the statement of the Lemma follows.
\end{proof}

In terms of Proposition \ref{prop:lambda_convex_hull} it now remains to prove the inclusion $\overline{U}\subset K^\Lambda$. The proof of this inclusion will rely on the Krein-Milman theorem for $\Lambda$-convex sets \cite[Lemma 4.16]{Kirchheim}. For this we discuss the following $\Lambda$-directions.
\begin{lemma}\label{lem:muskat_direction}
Let $z\in Z_0$. The element $\tilde{z}(z)\in Z$ defined by 
\begin{gather*}
\tilde{\rho}(z):=1,\quad
\tilde{v}(z):=\frac{m-\rho v}{1-\rho^2},\quad
\tilde{m}(z):=v-\rho\tilde{v}(z),
\\
\tilde{\sigma}(z)+\tilde{p}(z)\id:=\tilde{m}(z)\otimes \tilde{v}(z)+\tilde{v}(z)\otimes \tilde{m}(z)
\end{gather*}
is contained in $\Lambda$ and has the property that for every $t\in(-1-\rho,1-\rho)$ there holds
\begin{align*}
\tilde{z}(z+t\tilde{z}(z))=\tilde{z}(z),\quad T_\pm(z+t\tilde{z}(z))=T_\pm(z),\quad M(z+t\tilde{z}(z))^\circ=M(z)^\circ.
\end{align*}
\end{lemma}
\begin{proof}
The proof is a straightforward adaption of Lemma 4.6 (ii),(iii) in \cite{GKSz} and therefore only sketched here. As a nontrivial element $(\xi,c)\in\R^n\times\R$ in the kernel of $M_\Lambda(\tilde{z}(z))$
one can take any $\xi\neq 0$ contained in the orthogonal complement of $\tilde{v}(v)$ and set $c=-\tilde{m}(z)\cdot\xi$.

The stated invariances can be verified directly. Note that for $T_\pm$ it helps to rewrite
\[
T_\pm(z)=\frac{1}{n}\abs{v+(\pm 1-\rho)\tilde{v}(z)}^2,
\]
whereas for $M^\circ$ identity \eqref{eq:Ubdd2} is useful.
\end{proof}
As in \cite{GKSz} we call $\tilde{z}(z)$ the Muskat direction associated with $z$, since it generalizes the density perturbation of the Muskat problem introduced in \cite{Sz-Muskat}. 
Also as in \cite{GKSz} we have the following lemma concering Euler type directions preserving the density.
\begin{lemma}\label{lem:euler_directions}
For any pair $(\bar{v},\bar{\sigma})\in \R^n\times  \mathcal S_0^{n \times n}$, $\bar{v}\neq0$, there exists $\bar{p}\in\R$, such that for all $\lambda\in\R$ the vector $\bar{z}_\lambda:=(0,\bar{v},\lambda\bar{v},\bar{\sigma},\bar{p})$ belongs to $\Lambda$. Moreover, for all $t\in\R$ there holds
\[
T_+(z+t\bar{z}_{-1})=T_+(z),\quad T_-(z+t\bar{z}_{+1})=T_-(z).
\]
\end{lemma}
\begin{proof}
See the proof of Lemma 4.6 (i),(iv) in \cite{GKSz}.
\end{proof}

We have the following results concerning $\Lambda$-extreme points of $\overline{U}$.
Recall that $\pi:Z\rightarrow\R\times\R^n\times\R^n\times \cS_0^{n\times n}$ is the projection from \eqref{eq:projection}.
\begin{lemma}\label{lem:boundedness_of_U_extreme_points}
The set $\pi(U)$ is bounded by a constant depending only on $e$ and the dimension $n$. Moreover, for every $z\in\overline{U}\setminus K$ there exists $\bar{z}\in \Lambda\setminus\{0\}$, such that $z\pm\bar{z}\in \overline{U}$.
\end{lemma}
\begin{proof}
Let $z\in U$. Clearly $\abs{\rho}\leq 1$ and  the two inequalities \eqref{eq:Ubdd1} imply a bound on $v$ and $m$ in terms of $e$ and $n$.
Using \eqref{eq:Ubdd2}, \eqref{eq:Ubdd3}, we obtain that
$M(z)+\sigma$ is also bounded by means of $e$ and $n$. In consequence we obtain $\abs{\tr M(z)}\leq c(e,n)$. Since the trace is bounded and  $\lamax(M(z))=Q(z)<e[\rho]$, using that $z\in U$, we get a corresponding bound on the whole spectrum of $M(z)$. Hence, $M(z)+\sigma$ and $M(z)$ are both uniformly bounded, and therefore $\abs{\sigma}\leq c(e,n)$. This proves that $\pi(U)$ is bounded.

Next we turn to the perturbation property. Let $z\in\overline{U}\setminus K$ and recall from Lemma \ref{lem:U_convex_closure_of_U} that $\overline{U}=\overline{U}_0\cup K_+'\cup K_-'$, $K\subset K_+'\cup K_-'$. 

If $z\in K_+'\setminus K$, there exists an Euler type direction $\bar{z}_{+1}$, i.e. $\bar{m}=\bar{v}$, from Lemma \ref{lem:euler_directions}, such that $z\pm \bar{z}_{+1}\in K_+'\setminus K$. The proof is the same as in \cite{DeL-Sz-Adm} and \cite[Lemma 4.8]{GKSz} and therefore omitted.
Similar for $z\in K_-'\setminus K$.

It remains to look at $z\in \overline{U}_0$. Let us first check in which cases we can use the associated Muskat direction $\bar{z}=\tilde{z}(z)$ from Lemma \ref{lem:muskat_direction}. By this Lemma the two inequalities $T_\pm(z+t\tilde{z}(z))\leq e[\pm 1]$ remain true for all $t\in (-1-\rho,1-\rho)$.
Furthermore, a straightforward computation shows that $Q(z)-e[\rho]$ can be rewritten as
\begin{align*}
Q(&z)-e[\rho]=\frac{1}{n}\tr M(z)+\lamax(M(z)^\circ)-e[\rho]\\
&=\frac{1-\rho}{2}T_-(z)+\frac{1+\rho}{2}T_+(z)+\lamax(M(z)^\circ)-\left(\frac{1-\rho}{2}e[-1]+\frac{1+\rho}{2}e[+1]\right).
\end{align*}
Using Lemma \ref{lem:muskat_direction} once more we therefore obtain
\begin{align*}
Q(z+t\tilde{z}(z))-e[\rho+t]=Q(z)-e[\rho]+\frac{t}{2}\big(T_+(z)-T_-(z)+e[-1]-e[+1]\big).
\end{align*}
Thus the desired inequality $Q(z+t\tilde{z}(z))\leq e[\rho+t]$ holds true for $\abs{t}>0$ sufficiently small in the case where $Q(z)<e[\rho]$, but also in the case where $Q(z)=e[\rho]$ and $T_+(z)-e[+1]=T_-(z)-e[-1]$.

Therefore it remains to treat the last case $Q(z)=e[\rho]$ and $T_+(z)-e[+1]\neq T_-(z)-e[-1]$. Note that this implies $\lamin (M(z))<e[\rho]$, since otherwise $e[\rho]=\lamax (M(z))=\lamin (M(z))$ yields $M(z)^\circ=0$ and thus 
\[
e[\rho]=Q(z)=\frac{1-\rho}{2}T_-(z)+\frac{1+\rho}{2}T_+(z).
\]
However, using that $T_\pm(z)\leq e[\pm 1]$, this equality can only hold if $T_\pm(z)= e[\pm 1]$, which is excluded in the considered case.

Let us assume $T_-(z)-e[-1]> T_+(z)-e[+1]$, the other case is treated similarly. 
We consider Euler directions from Lemma \ref{lem:euler_directions} such that $\bar m=\bar v$, i.e. $\bar{z}=\bar{z}_{+1}$ associated with $(\bar{v},\bar{\sigma})$ to be choosen. 
  By said Lemma such Euler directions  preserve $T_-$, i.e., $T_-(z+t\bar{z}_{+1})=T_-(z)\leq e[-1]$ for all $t\in\R$.

Once again proceeding as in \cite{DeL-Sz-Adm} or \cite[Lemma 4.8]{GKSz}, one may easily prove that there exists such an Euler direction which does not effect the maximal eigenvalue of $M(z)$, i.e. such 
 that $Q(z+t\bar{z})=Q(z)=e[\rho]$ for small enough $\abs{t}$. 
 The last condition needed for $z+t\bar{z}_{+1}\in\overline{U}$ follows from the continuity of $T_+$, i.e., for all $\abs{t}$ small enough one has
$
T_+(z+t\bar{z})-e[+1]< T_-(z)-e[-1]\leq 0.
$
\end{proof}
\begin{proof}[Proof of Proposition \ref{prop:lambda_convex_hull}]
From Lemma \ref{lem:U_convex_closure_of_U} one obtains $K^\Lambda\subset K^{co}\subset \overline{U}$, while Lemma \ref{lem:boundedness_of_U_extreme_points} implies that the $\Lambda$-extreme points of the up to the $p$-component compact set $\overline{U}$ are contained in $K$. The inclusion $\overline{U}\subset K^\Lambda$ follows from the Krein-Milman theorem for $\Lambda$-convex sets, cf. \cite[Lemma 4.16]{Kirchheim}.
\end{proof}

\subsection{Continuity of constraints}

We have the following result regarding the continuity of the nonlinear constraints $K_{(x,t)}$, given the continuity of the defining function $e(x,t)[\rho]$. This serves to have a set of subsolutions which is bounded in $L^2(\mathscr{D})$, where $\mathscr{D}:=\Omega\times (0,T)$. 

\begin{lemma}\label{lem:cont_of_constr}
Let $\mathscr{U}\subset\mathscr{D}$ be open and assume that the map $\mathscr{D}\times \R\rightarrow \R$, $(x,t,r)\mapsto e(x,t)[r]$ is continuous and bounded on $\mathscr{U}\times[-1,1]$, then it follows that the map $(x,t)\mapsto \pi(K_{(x,t)})$ is continuous and bounded on $\mathscr{U}$ with respect to the Hausdorff metric $d_{\mathcal{H}}$.
\end{lemma}
\begin{proof} The boundedness of $\bigcup_{(x,t)\in\mathscr{U}}\pi(K_{(x,t)})$ follows from Lemma \ref{lem:boundedness_of_U_extreme_points} and the boundedness of $e$.

Concerning the continuity let us fix $y:=(x,t)\in\mathscr{U}$ and $\varepsilon\in(0,1)$. In order to prove $d_{\mathcal{H}}(\pi(K_y),\pi(K_{y'}))<\varepsilon$ for all $y'=(x',t')\in B_\delta(y)\subset\mathscr{U}$ for a suitable $\delta=\delta(\varepsilon,y)>0$ we will use \cite[Lemma 3.1]{Crippa} saying that $d_{\mathcal{H}}(\pi(K_y),\pi(K_{y'}))<\varepsilon$ holds true provided for any $\pi(z)\in \pi(K_y)$ there exists $\pi(z')\in \pi(K_{y'})\cap B_\varepsilon(\pi(z))$ and vice versa.

First of all observe that by the continuity of $e$ there exists $\delta\in(0,\varepsilon)$ such that 
\begin{align}\label{eq:squareroot}
\abs{e(y)[\pm 1]-e(y')[\pm 1]}<\varepsilon,\quad\left|\left(ne(y)[\pm 1]\right)^{1/2}-\left(n e(y')[\pm 1]\right)^{1/2} \right|<\varepsilon,
\end{align} for any $y'\in B_\delta(y)\subset\mathscr{U}$.
Let now
$$
z=(\rho,v,\rho v, v\otimes v-e(y)[\rho]\id,p)\in K_y,
$$
with $\rho\in\{-1,1\}$ and $|v|^2=ne(y)[\rho]$. It follows 
that
$v=\left(ne(y)[\rho]\right)^{1/2}b,$
for some $b\in S^{n-1}$. For $y'\in B_\delta(y)$ we define 
$$
z':=(\rho,v',\rho v', v'\otimes v'-e(y')[\rho]\id,p)
$$ by setting 
$
v':=\left(ne(y')[\rho]\right)^{1/2}b.
$
Note that $z'\in K_{y'}$.

Furthermore, from \eqref{eq:squareroot} it follows that
\begin{align*}
|v-v'|<\varepsilon,\quad |m-m'|<\varepsilon,\quad |\sigma-\sigma'|<(n+1)\varepsilon.
\end{align*}
This way we have shown that for any $y'\in B_\delta(y)$ and any $z\in K_y$ there exists $z'\in K_{y'}\cap B_{c\varepsilon}(z)$ for some $c>0$ depending only on the dimension $n$. Using the symmetry of this construction, one can similarly prove that for any $z'\in K_{y'}$ there exists $z\in K_y$ such that $|z-z'|<c\varepsilon$. As illustrated above we then conclude $d_{\cH}(\pi(K_y),\pi(K_{y'}))<c\varepsilon$ via \cite[Lemma 3.1]{Crippa}.
\end{proof}

\subsection{Conclusion}

We have now collected all the ingredients for the proof of Theorem \ref{thm:main2}, which follows by the known convex integration procedures
in the Tartar framework \cite{DeL-Sz-Annals,DeL-Sz-Adm} and its refinements \cite{Castro-Faraco-Mengual,Crippa}. We refrain from  formulating another version of the Tartar framework exactly taylored to our needs and instead only point out the small modifications that need to be done in the existing convex integration theorems in order to conclude Theorem \ref{thm:main2}. 

We begin with the functional setup.
Let $\mathscr{D}:=\Omega\times (0,T)$. Fix a function $e:\mathscr{D}\times [-1,1]\rightarrow \R$, a subsolution $\hat{z}=(\hat{\rho},\hat{v},\hat{m},\hat{\sigma},\hat{p})$ with initial data $(\rho_0,v_0)$ and mixing zone $\mathscr{U}$, as well as an error function $\delta:[0,T]\rightarrow \R$ as stated in Theorem \ref{thm:main2}.
Define $X_0$ to be the set of all functions $\pi(z)=(\rho,v,m,\sigma)$, such that
\begin{itemize}
\item $z=(\rho,v,m,\sigma,p)$ is a subsolution for $e$, $(\rho_0,v_0)$ and with the same mixing zone $\mathscr{U}$, in the sense of Definition \ref{def:subsolEuler},
\item $z=\hat{z}$ a.e. on $\mathscr{D}\setminus\mathscr{U}$,
\item there exists $C(z)\in (0,1)$, such that for all $t\in[0,T]$ there holds
\begin{align}\label{eq:X_0_bonus_property}
\abs{\int_\Omega\left(\frac{n}{2}e_1(x,t)+gAx_n\right)(\hat{\rho}(x,t)-\rho(x,t))\:dx}\leq C(z)\delta(t).
\end{align}
\end{itemize}
Recall from Section \ref{sec:statement_relaxation} that $e(x,t)[r]=e_0(x,t)+re_1(x,t)$ with $L^\infty$ functions $e_0,e_1$, where $e_1$ is additionally of class $\cC^0([0,T];L^2(\Omega))$.

Next we will equip $X_0$ with a suitable metric. Recall from Remark \ref{rem:regulartiy_of_rho} that for any $\pi(z)\in X_0$ there holds 
  $\rho\in\cC^0([0,T];L^2_w(\Omega))$. Moreover, for every element from $X_0$ there holds$\norm{\rho(\cdot,t)}^2_{L^2(\Omega)}\leq \abs{\Omega}$, $t\in[0,T]$ and $\norm{(v,m,\sigma)}_{L^2(\mathscr{D})}^2\leq c\abs{\Omega}T$ for a constant $c$ depending only on $\norm{e_0}_{L^\infty(\mathscr{D})},\norm{e_1}_{L^\infty(\mathscr{D})}$ and the dimension $n$. This is due to Lemma \ref{lem:boundedness_of_U_extreme_points}.

Thus we can find two bounded closed balls $B^{(1)}$ contained in $L^2(\Omega)$ and  $B^{(2)}$ contained in $L^2(\mathscr{D};\R^n\times\R^n\times\cS_0^{n\times n})$, such that every function $\pi(z)\in X_0$ satisfies $\rho(\cdot,t)\in B^{(1)}$, $t\in[0,T]$, $(v,m,\sigma)\in B^{(2)}$. As in \cite{Castro-Faraco-Mengual,DeL-Sz-Adm} let $d^{(i)}$, $i=1,2$, be a metric on $B^{(i)}$ metrizing the corresponding weak $L^2$-topology and define for $\pi(z),\pi(z')\in X_0$ the metric
\[
d_X(\pi(z),\pi(z')):=\sup\set{\sup_{t\in[0,T]}d^{(1)}\big(\rho(\cdot,t),\rho'(\cdot,t)\big),d^{(2)}\big((v,m,\sigma),(v',m',\sigma')\big)}.
\]
Finally let $X$ be the closure of $X_0$ in $\cC^0([0,T];(B^{(1)},d^{(1)}))\times (B^{(2)},d^{(2)})$ with respect to the metric $d_X$. Then $X$ is a complete metric space with $d_X(\pi(z_j),\pi(z))\rightarrow 0$ if and only if $\rho_j\rightarrow \rho$ in $\cC^0([0,T];L^2_w(\Omega))$ and $(v_j,m_j,\sigma_j)\rightharpoonup (v,m,\sigma)$ weakly in $L^2(\mathscr{D};\R^n\times\R^n\times \cS_0^{n\times n})$. Concerning notation we again denote elements from $X$ by $\pi(z)$.

Note that the $d_X$ topology is stronger than the topology coming from simply metrizing the weak topology on a bounded closed ball of $L^2(\mathscr{D};\pi(Z))$. In consequence the functional $I:X\rightarrow \R$,
\begin{align}\label{eq:definition_of_I}
I(\pi(z)):=\int_{\mathscr{D}}\abs{\pi(z(x,t))}^2\:d(x,t)
\end{align}
is still a Baire-1 functional, cf. \cite[Section 2.3]{Crippa}. We also define $J:X\rightarrow \R$,
\begin{align}\label{eq:definition_of_J}
J(\pi(z)):=\int_{\mathscr{D}}\dist\big(\pi(z(x,t)),\pi(K_{(x,t)})\big)^2\:d(x,t).
\end{align}
Note that $J$ is continuous with respect to the strong $L^2(\mathscr{D};\pi(Z))$ topology.
\begin{lemma}[Perturbation Lemma]\label{lem:perturbation_lemma}
Let $\alpha>0$. There exists $\beta>0$, such that for every $\pi(z)\in X_0$ with $J(\pi(z))\geq \alpha$ there exists a sequence $(\pi(z_k))_{k\in\N}\subset X_0$ with $d_X(\pi(z_k),\pi(z))\rightarrow 0$ and such that for all $k\in\N$ there holds
\begin{align}\label{eq:perturbation_property}
\int_{\mathscr{D}}\abs{\pi(z_k(x,t))-\pi(z(x,t))}^2\:d(x,t)\geq \beta.
\end{align}
\end{lemma}
\begin{proof}
If we neglect for now property \eqref{eq:X_0_bonus_property} in the definition of $X_0$, it follows as in \cite[Lemma 2.4]{Crippa} from Lemmas \ref{lem:density_of_lambda_prime}, \ref{lem:locpw}, \ref{lem:cont_of_constr} and Corollary \ref{cor:long_segments} that there exists $\beta(\alpha)>0$ and a sequence $(\pi(z_k))_{k\in\N}\subset X_0\setminus\{\eqref{eq:X_0_bonus_property}\}$ satisfying \eqref{eq:perturbation_property} and $\pi(z_k)\rightharpoonup \pi(z)$ weakly in $L^2(\mathscr{D};\pi(Z))$. At this point the only difference that prevents us from citing \cite[Lemma 2.4]{Crippa} literally is the projection $\pi$, but as in \cite[Lemma 5.3]{GKSz} the projection can be included by canonical modifications.

It therefore remains to improve the convergence of the $\rho$-component from $\rho_k\rightharpoonup \rho$ weakly in $L^2(\mathscr{D})$ to $\rho_k\rightarrow \rho$ in $\cC^0([0,T];L_w^2(\Omega))$ and to show that the functions $(\pi(z_k))_{k\in\N}$ satisfy \eqref{eq:X_0_bonus_property} for all $k$ big enough.
However, the improved convergence follows from Remark \ref{rem:slightly_stronger_convergence_of_plane_waves} by using cylinders instead of balls in the proof of \cite[Lemma 2.4]{Crippa}.
Finally, the fact that the sequence $(\pi(z_k))_{k\in\N}$ satisfies property \eqref{eq:X_0_bonus_property} for $k$ big enough follows as in Step 3 of the proof of \cite[Proposition 3.1]{Castro-Faraco-Mengual}. Indeed, since $z\in X_0$, it is enough to fix $C'(z)\in(0,1-C(z))$ and to show
\[
\abs{\int_\Omega\left(\frac{n}{2}e_1(x,t)+gAx_n\right)(\rho_k(x,t)-\rho(x,t))\:dx}\leq C'(z)\delta(t)
\]
for all $t\in[0,T]$ and all $k$ sufficiently large. Since by construction $\rho_k=\rho$ outside a compact subset of the mixing zone $\mathscr{U}$, hence outside a set contained in $[t_0,t_1]\times \Omega$ for some $0<t_0<t_1<T$, it is enough to show 
\[
\forall t\in[t_0,t_1]:\quad\abs{\int_\Omega f(x,t)(\rho_k(x,t)-\rho(x,t))\:dx}\leq C'(z)\delta_0,
\]
where $f(x,t):=\frac{n}{2}e_1(x,t)+gAx_n$ and $\delta_0:=\inf\set{\delta(t)>0:t\in[t_0,t_1]}>0$. But the latter inequality holds true for big enough $k$ due to the uniform continuity of the map $[0,T]\ni t\mapsto f(\cdot,t)\in L^2(\Omega)$, the uniform bound on $\norm{\rho_k(\cdot,t)}_{L^2(\Omega)}$ and the convergence $\rho_k\rightarrow \rho$ in $\cC^0([0,T];L^2_w(\Omega))$.
\end{proof}
\begin{proof}[Proof of Theorem \ref{thm:main2}]
Having Lemma \ref{lem:perturbation_lemma} at hand we can prove as in \cite{Crippa} or \cite{GKSz} that $J^{-1}(0)$ is contained in the set of continuity points of $I$, where $I,J$ were defined in \eqref{eq:definition_of_I}, \eqref{eq:definition_of_J}. Since $I$ is Baire-1, this shows that $J^{-1}(0)$ is residual in $(X,d_X)$. Observe also that if $\pi(z)\in J^{-1}(0)$, then $(\rho,v)$ is a weak solution of \eqref{eq:bou}, \eqref{eq:boundary_condition} satisfying properties a) and b) of Theorem \ref{thm:main2}. 

Concering property Thm. \ref{thm:main2} c), approximation by elements from $X_0$ with respect to $d_X$ shows that any element $\pi(z)$ from $X$ satisfies 
\[
\abs{\int_\Omega\left(\frac{n}{2}e_1(x,t)+gAx_n\right)(\hat{\rho}(x,t)-\rho(x,t))\:dx}\leq \delta(t)
\]
for all $t\in[0,T]$.

Finally property \ref{thm:main2} d) is a consequence of \cite[Corollary 3.1]{Castro-Faraco-Mengual}. 
This finishes the proof of Theorem \ref{thm:main2}.
\end{proof}
\section{Subsolutions}\label{sec:subsolutions}
Let us turn to the construction of subsolutions on the $n$-dimensional box $\Omega:=(0,1)^{n-1}\times(-L,L)$, $L>0$ with Rayleigh-Taylor initial data \eqref{eq:initial_data}. Let $T>0$ and $\mathscr{D}:=\Omega\times (0,T)$. Neglecting the admissibility, recall from Definition \ref{def:subsolEuler} that a subsolution $z=(\rho,v,m,\sigma,p)$ is a weak solution of the linear system 
\eqref{eq:subsolution_system_sec2}
on $\mathscr{D}$ with 
boundary data
\eqref{eq:subsolution_boundary_data_sec2}
which is continuous on an open subset $\mathscr{U}\subset\mathscr{D}$ satisfying
\begin{gather}\label{eq:hull_inequalities}
\begin{gathered}
\rho\in(-1,1),\quad \frac{\abs{m\pm v}^2}{n(1\pm \rho)^2}<e[\pm 1],\\
\lamax\left(\frac{v\otimes v-\rho(m\otimes v+v\otimes m)+m\otimes m}{1-\rho^2}-\sigma\right)<e[\rho]
\end{gathered}
\end{gather}
there, where $e:\mathscr{D}\times \R\rightarrow\R$, $(x,t,r)\mapsto e(x,t)[r]$ is continous on $\mathscr{U}$ and affine with respect to $r$. Outside of $\mathscr{U}$ the conditions $\rho\in\set{\pm 1}$, $v\otimes v-\sigma=e[\rho]\id_{\R^n}$, $m=\rho v$ are required to hold almost everywhere.

Due to the heuristic argument in Section \ref{sec:choices_concerning_the_energy}, we consider $e=e_\varepsilon$ to be of the form 
\begin{align*}
e_\varepsilon(x,t)[r]=\tilde{e}_\varepsilon(x,t)-\varepsilon gA x_n r
\end{align*}
with $\tilde{e}_\varepsilon:\mathscr{D}\rightarrow\R$ continuous on $\mathscr{U}$ and $\varepsilon\in \left[0,\frac{2}{n}\right]$, such that Theorem \ref{thm:main2} will produce turbulent solutions to \eqref{eq:bou},  \eqref{eq:boundary_condition}, \eqref{eq:initial_data} with local energy given by
\begin{align*}
\cE_{sol}(x,t)=\frac{n}{2}\tilde e_\varepsilon(x,t)+\left(1-\frac{n}{2}\varepsilon\right)\rho_{sol}(x,t)gAx_n.
\end{align*}

\subsection{Self-similar subsolutions}\label{sec:1Dsubsols}

In this section we prove Lemma \ref{lem:1D_subsolutions}. 
Recall the definitions of $\mathcal F$ and $\mathcal A$ from \eqref{eq:definition_of_set_F}, respectively \eqref{eq:definition_of_set_A}.
\begin{proof}[Proof of Lemma \ref{lem:1D_subsolutions}]
For
$f\in \cF$ define $F:[-1,1]\rightarrow \R$,
\[
F(y):=\int_{-1}^yf'(s)s\:ds.
\]
For any choice of a profile $f\in \cF$ and a growth rate $a\in \cA$ one can check that
 $z=(\rho,v,m,\sigma,p):\mathscr{D}\rightarrow Z$ defined by $v\equiv 0$, 
$\rho(x,t)=1$, $m(x,t)=0$ for $x_n\geq a(t)$, $\rho(x,t)=-1$, $m(x,t)=0$ for $x_n\leq -a(t)$ and 
\[
\rho(x,t)=f\left(\frac{x_n}{a(t)}\right),\quad m(x,t)=\dot{a}(t)F\left(\frac{x_n}{a(t)}\right)e_n
\]
for $x_n\in (-a(t),a(t))$, as well as $\sigma(x,t)=0$ for $\abs{x_n}\geq a(t)$,
\begin{align}\label{eq:subsp}
\begin{split}
\sigma(x,t)=\frac{\left(m(x,t)\otimes m(x,t)\right)^\circ}{1-\rho(x,t)^2}\quad \text{ for }\abs{x_n}<a(t),\\
p(x,t)=-\sigma_{nn}(x,t)-gA\int_{-L}^{x_n}\rho(\tilde{x}_n,t)\:d\tilde{x}_n
\end{split}
\end{align}
are continuous on $\overline{\mathscr{D}}\setminus \big(\R^{n-1}\times\{0\}\times\{0\}\big)$ piecewise $\cC^1$ and satisfy \eqref{eq:subsolution_system_sec2}, \eqref{eq:initial_data}, and also \eqref{eq:subsolution_boundary_data_sec2} as long as $a(t)\leq L$ for all $t\in(0,T)$. The continuity of $m$ is a consequence of the symmetry of $f$, while the continuity of $\sigma$ follows by an expansion at the points $x_n=\pm a(t)$ and the condition $f'(\pm 1)> 0$.

Once the construction of the subsolution is finished the set
\[
\mathscr{U}:=\set{(x,t)\in\mathscr{D}:x_n\in(-a(t),a(t))}
\]
will be the mixing zone. Concerning the pointwise constraints we define 
\begin{align}\label{eq:subse}
\begin{split}
\tilde{e}_\varepsilon(x,t):=&\max\set{\frac{m_n(x,t)^2}{n(1+\rho(x,t))^2}+\varepsilon gAx_n,\frac{m_n(x,t)^2}{n(1-\rho(x,t))^2}-\varepsilon gAx_n}\\
&\hspace{215pt}+\left(1-\rho(x,t)^2\right)\delta(x,t)
\end{split}
\end{align}
for $(x,t)\in \mathscr{U}$ and $\tilde{e}_\varepsilon(x,t)=\varepsilon gA \abs{x_n}$ for $(x,t)\in\overline{\mathscr{D}}\setminus\mathscr{U}$. Here $\delta:\mathscr{D}\rightarrow(0,+\infty)$ is a continuous, even, positive and typically small function guaranteeing the inequalities \eqref{eq:hull_inequalities} to hold in a strict sense.
Indeed the first three conditions in \eqref{eq:hull_inequalities} hold by definition of $\rho$ and $\tilde{e}_\varepsilon$. For the last inequality we have
\begin{align*}
\lamax\left(\frac{m\otimes m}{1-\rho^2}-\sigma\right)&=\frac{\abs{m}^2}{n(1-\rho^2)}=\frac{1+\rho}{2}\frac{\abs{m}^2}{n(1+\rho)^2}+\frac{1-\rho}{2}\frac{\abs{m}^2}{n(1-\rho)^2}\\
&< \frac{1+\rho}{2}e_\varepsilon[+1]+\frac{1-\rho}{2}e_\varepsilon[-1]=e_\varepsilon[\rho].
\end{align*}
Outside of $\mathscr{U}$ it is clear that $\rho=1$ on $\set{x_n\geq a(t)}$, $\rho=-1$ on $\set{x_n\leq -a(t)}$,  $m=0=\rho v$ and $v\otimes v-\sigma=0=\tilde{e}_\varepsilon-\varepsilon gA \abs{x_n}=e_\varepsilon[\rho]$.
This concludes the proof of Lemma \ref{lem:1D_subsolutions}.
\end{proof}

\subsection{Admissibility and maximal initial energy dissipation}\label{sec:admissibility_of_subsols}
Instead of investigating all admissible subsolutions emanating from Section \ref{sec:1Dsubsols}, we will focus  on the one that is selected by asking for maximal initial energy dissipation. 

For $(f,a,\varepsilon)\in\cF\times \cA\times\left[0,\frac{2}{n}\right]$ observe that the total energy at time $t>0$ of the induced subsolution can be choosen arbitrarily close to $E_{f,a,\varepsilon}(t)$ defined in
\eqref{eq:subste}, which for admissibility has to be less than the initial energy $E(0)=\int_\Omega gA\abs{x_n}\:dx$. In fact $E_{f,a,\varepsilon}(t)$ can be obtained from $\int_\Omega \frac{n}{2}\tilde{e}_\varepsilon(x,t)+\left(1-\frac{n}{2}\varepsilon\right)gAx_n\rho(x,t)\:dx$ with $\tilde{e}_\varepsilon$ defined in \eqref{eq:subse} by letting $\delta\to 0$ in $L^\infty((0,T);L^1(\Omega))$.

Using this, the definitions of $\rho$, $m$ from Section \ref{sec:1Dsubsols}, the transformation $x_n=a(t)y$ and the symmetry of $f$ one sees that the difference of the energies can be computed by the following integrals
\begin{align*}
&E_{f,a,\varepsilon}(t)-E(0)\\&\hspace{10pt}=2\int_0^1\max\set{\frac{a(t)\dot{a}(t)^2F(y)^2}{2(1+f(y))^2}+\frac{n}{2}\varepsilon gAa(t)^2y,\frac{a(t)\dot{a}(t)^2F(y)^2}{2(1-f(y))^2}-\frac{n}{2}\varepsilon gAa(t)^2y}\:dy\\&
\hspace{45pt}+a(t)^2gA\int_0^1(2-n\varepsilon)yf(y)-2y\:dy.
\end{align*}
Concerning the well-definedness observe again that for all $f\in\cF$ the quotient $\frac{F(y)}{1-f(y)}$ has a finite limit as $y\rightarrow 1$.

For a given profile $f\in\cF$ and a growth rate $a\in\cA$ one can via the above formula simply check by hands the admissibility of the induced self-similar subsolution.
\begin{example}\label{ex:ss}
If $T\leq \sqrt{\frac{3L}{gA}}$, the choices $f(y)=y$, 
 $a(t)=\frac{1}{3}gAt^2$ and $\varepsilon=\frac{2}{3n}$ give rise to a subsolution on $\Omega\times (0,T)$ with
\[
E_{f,a,\varepsilon}(t)-E(0)=-\frac{g^3A^3}{81}t^4.
\]
In particular this implies that the subsolution is admissible for small $\delta(x,t)$. 
\end{example}
\begin{remark}\label{rem:energy_ration}
The released potential energy of the subsolution above at time $t\in[0,T)$ is given by
\[
\int_{\Omega}gAx_n\rho(x,t)\:dx-E(0)=-\frac{g^3A^3}{27}t^4.
\]
Therefore the ratio between dissipated and released energy is $\frac{1}{3}$.
\end{remark}

Besides the fact of being a simple example, it turns out that these choices for $f,a,\varepsilon$ maximize the initial energy dissipation.

Recall the functionals $J_k$, $k=0,\ldots,4$ from \eqref{eq:Jk}.
Since $a(0)=0$, there clearly holds $J_0(f,a,\varepsilon)=0$. 
We are now in position to prove Theorem \ref{thm:selection_by_initial_dissipation}.

\begin{proof}[Proof of Theorem \ref{thm:selection_by_initial_dissipation}]

In the formula for the energy difference let us abbreviate the two terms among which the maximum is taken, i.e., set
\begin{align}\label{eq:definition_of_terms_under_max}
\begin{split}
G^+_{f,a,\varepsilon}(y,t)&:=\frac{a(t)\dot{a}(t)^2F(y)^2}{2(1+f(y))^2}+\frac{n}{2}\varepsilon gAa(t)^2y,\\
G^-_{f,a,\varepsilon}(y,t)&:=\frac{a(t)\dot{a}(t)^2F(y)^2}{2(1-f(y))^2}-\frac{n}{2}\varepsilon gAa(t)^2y.
\end{split}
\end{align}
Estimating the maximum from below by the convex combination
\begin{align}\label{eq:estimating_max_by_convex_combination}
\begin{split}
\max\set{G^+_{f,a,\varepsilon}(y,t),G^-_{f,a,\varepsilon}(y,t)}&\geq \frac{1+f(y)}{2}G^+_{f,a,\varepsilon}(y,t)+\frac{1-f(y)}{2}G^-_{f,a,\varepsilon}(y,t)\\
&=\frac{a(t)\dot{a}(t)^2F(y)^2}{2(1-f(y)^2)}+\frac{n}{2}f(y)\varepsilon gA a(t)^2y
\end{split}
\end{align}
yields
\begin{align*}
J_k(f,a,\varepsilon)&\geq \lim_{t\rightarrow 0}\left(a(t)\dot{a}(t)^2\int_0^1\frac{F(y)^2}{1-f(y)^2}\:dy-2a(t)^2gA\int_0^1(1-f(y))y\:dy\right)t^{-k}.
\end{align*}
Observe that
\begin{align}\label{eq:definition_of_Ii}
I_1(f):=\int_0^1\frac{F(y)^2}{1-f(y)^2}\:dy>0,\quad I_2(f):=\int_0^1(1-f(y))y\:dy>0,
\end{align}
such that the required admissibility implies
\[
0\geq J_1(f,a,\varepsilon)\geq \dot{a}(0)^3I_1(f)\geq 0,
\]
and therefore $\dot{a}(0)=0$, $J_1(f,a,\varepsilon)=0$. Since now $a(t)=\frac{1}{2}\ddot{a}(0)t^2+o(t^2)$ as $t\rightarrow 0$ the admissibility also implies $J_2(f,a,\varepsilon)=J_3(f,a,\varepsilon)=0$. This proves the first part of the Theorem.

The lowest order for which the initial energy dissipation rate is not necessarily vanishing is $4$. There holds
\begin{align}\label{eq:estimating_J4}
J_4(f,a,\varepsilon)\geq \frac{1}{2}\ddot{a}(0)^3I_1(f)-\frac{1}{2}\ddot{a}(0)^2gAI_2(f)=:\tilde{J}(f,\ddot{a}(0)).
\end{align}
In Lemma \ref{lem:unique_minimizer} below we will show that the functional $\tilde{J}:\cF\times [0,+\infty)\rightarrow \R$ has a unique global minimum in $f(y)=y$ and $\ddot{a}(0)=\frac{2}{3}gA$.

It follows that for any $(f,a,\varepsilon)\in\cF\times\cA\times \left[0,\frac{2}{n}\right]$ leading to an admissible subsolution there holds
\[
J_4(f,a,\varepsilon)\geq \tilde{J}\left(\id,\frac{2}{3}gA\right)=-\frac{g^3A^3}{81}.
\]
Note that $I_1(\id)=\frac{1}{6}$, $I_2(\id)=\frac{1}{6}$.
It remains to check that this lower bound is achieved for $f(y)=y$, any $a\in\cA$ with $a(t)=\frac{1}{3}gAt^2+o(t^2)$ and $\varepsilon=\frac{2}{3n}$. This is a consequence of the fact that for this choice the two limits
\begin{align*}
\lim_{t\rightarrow 0}\frac{G^\pm_{f,a,\varepsilon}(y,t)}{t^4}=\frac{\ddot{a}(0)^3F(y)^2}{4(1\pm f(y))^2}\pm\frac{n}{8}\varepsilon gA\ddot{a}(0)^2y=\frac{g^3A^3}{54}(1+y^2),
\end{align*}
with $G^\pm_{f,a,\varepsilon}$ defined in \eqref{eq:definition_of_terms_under_max}, coincide. Therefore instead of an inequality we actually have equality when dividing \eqref{eq:estimating_max_by_convex_combination} by $t^4$ and passing to the limit $t\rightarrow 0$. Thus we also have equality in \eqref{eq:estimating_J4}, which means 
\[
J_4\left(\id,\frac{1}{3}gAt^2+o(t^2),\frac{2}{3n}\right)=\tilde{J}\left(\id,\frac{2}{3}gA\right)=-\frac{g^3A^3}{81}.
\]
The uniqueness of the minimizer follows from the uniqueness of the minimizer of $\tilde{J}$ and the fact that for $f(y)=y$, $a(t)=\frac{1}{3}gt^2+o(t^2)$, any choice of $\varepsilon\neq \frac{2}{3n}$ leads to a strict inequality when estimating the maximum by the convex combination in the limit $t\rightarrow 0$ of  $\frac{\eqref{eq:estimating_max_by_convex_combination}}{t^4}$.
\end{proof}

\begin{lemma}\label{lem:unique_minimizer}
The functional $\tilde{J}:\cF\times[0,+\infty)\rightarrow\R$,
\[
\tilde{J}(f,c)=\frac{1}{2}c^3I_1(f)-\frac{1}{2}c^2gAI_2(f)
\]
with $I_{1,2}(f)$ defined in \eqref{eq:definition_of_Ii} has a unique global minimum in $\left(\id,\frac{2}{3}gA\right)$.
\end{lemma}
\begin{proof}

First of all observe that for fixed $f\in\cF$ the function $\tilde{J}(f,\cdot):[0,+\infty)\rightarrow \R$ has a unique minimum in $c_0(f)=\frac{2}{3}gA\frac{I_2(f)}{I_1(f)}$.
Therefore
\begin{align*}
\tilde{J}(f,c)\geq \tilde{J}(f,c_0(f))=-\frac{2}{27}g^3A^3\frac{I_2(f)^3}{I_1(f)^2}
\end{align*}
and it remains to show 
\begin{align}\label{eq:final_inequality}
6I_2(f)^3<  I_1(f)^2
\end{align}
for any $f\in\cF\setminus \{\id\}$. Note that for $f=\id$ there holds equality, since $I_1(\id)=\frac{1}{6}$, $I_2(\id)=\frac{1}{6}$.

Let us rewrite
\begin{align*}
I_2(f)=\int_0^1(1-f(y))y\:dy=\int_0^1f'(y)\frac{1}{2}y^2\:dy=\frac{1}{2}\int_0^1yF'(y)\:dy=-\frac{1}{2}\int_0^1F(y)\:dy.
\end{align*}
Since $I_2(f)>0$, inequality \eqref{eq:final_inequality} is equivalent to $6I_2(f)^4<I_1(f)^2I_2(f)$. Now
\begin{align*}
6I_2(f)^4=\frac{3}{8}\left(\int_0^1\frac{F(y)}{\sqrt{1-f(y)^2}}\sqrt{1-f(y)^2}\:dy\right)^4\leq \frac{3}{8}I_1(f)^2\left(\int_0^11-f(y)^2\:dy\right)^2.
\end{align*}
Since also $I_1(f)$ is positive, we see that \eqref{eq:final_inequality} holds true provided
\begin{align*}
\hat{J}(f):=-\int_0^1F(y)\:dy-\frac{3}{4}\left(\int_0^11-f(y)^2\:dy\right)^2>0.
\end{align*}
In order to prove $\hat{J}(f)>0$ for $f\in\cF\setminus\{\id\}$ we write $f=\id+\varphi$ with $\varphi\neq 0$, such that
\[
F(y)=\int_{-1}^y(1+\varphi'(s))s\:ds=\frac{1}{2}(y^2-1)+y\varphi(y)-\int_{-1}^y\varphi(s)\:ds
\]
and
\begin{align*}
\hat{J}(\id+\varphi)&=-\int_0^1\frac{1}{2}(y^2-1)+y\varphi(y)-\int_{-1}^y\varphi(s)\:ds\:dy\\
&\hspace{60pt}-\frac{3}{4}\left(\int_0^11-y^2-2y\varphi(y)-\varphi(y)^2\:dy\right)^2\\
&=\int_0^1\varphi(y)^2\:dy-\frac{3}{4}\left(\int_0^12y\varphi(y)+\varphi(y)^2\:dy\right)^2.
\end{align*}
Thus in terms of $f$ and the $L^2(0,1)$ inner product and norm we have
\begin{align}\label{eq:definition_of_hat_J}
\hat{J}(f)=\norm{f-\id}_{L^2(0,1)}^2-\frac{3}{4}\ska{f-\id,f+\id}_{L^2(0,1)}^2.
\end{align}
The next (and last for this subsection) lemma implies that $\hat{J}(f)>0$ for all $f\in\cF\setminus\{\id\}$, which allows us to conclude the proof of Lemma \ref{lem:unique_minimizer}.
\end{proof}
\begin{lemma}\label{lem:reduced_functional}
Let $\cF_0:=\set{f\in L^2(0,1):\abs{f}<1\text{ a.e.}}$. The functional $\hat{J}$ defined in \eqref{eq:definition_of_hat_J} satisfies $\hat{J}(f)>0$ for all $f\in\cF_0\setminus\{\id\}$.
\end{lemma}
\begin{proof}
We set $\overline{\cF}_0:=\set{f\in L^2(0,1):\abs{f}\leq 1\text{ a.e.}}$, which is the closure of $\cF_0$ with respect to $\norm{\cdot}_{L^2(0,1)}$, and observe that $\hat{J}(f)\geq -\frac{4}{3}$ for $f\in\overline{\cF}_0$. Now let $(f_n)_{n\in\N}\subset\overline{\cF}_0$ be a minimzing sequence for $\hat{J}$. Since $\overline{\cF}_0$ is bounded and convex there exists $f_*\in\overline{\cF}_0$ with $f_n\rightharpoonup f_*$ in $L^2(0,1)$ along a subsequence. By the weak lower semicontinuity of the norm and since 
\begin{equation}\label{eq:form_of_hat_J}
\hat{J}(f)=h\left(\norm{f}_{L^2(0,1)}^2\right)-2\ska{\id,f}_{L^2(0,1)}
\end{equation}
with $h:[0,1]\rightarrow \R$, 
\begin{align*}
h(x)=x+\frac{1}{3}-\frac{3}{4}\left(x-\frac{1}{3}\right)^2,\quad h'(x)=\frac{3}{2}(1-x)\geq 0,
\end{align*}
there holds 
\begin{align*}
\inf_{\overline{\cF}_0}\hat{J}&=\liminf_{n\rightarrow+\infty}\hat{J}(f_n)=h\left(\liminf_{n\rightarrow+\infty}\norm{f_n}_{L^2(0,1)}^2\right)-\liminf_{n\rightarrow+\infty}2\ska{\id,f_n}_{L^2(0,1)}\\
&\geq h\left(\norm{f_*}^2_{L^2(0,1)}\right)-2\ska{\id,f_*}_{L^2(0,1)}=\hat{J}(f_*).
\end{align*}
Thus the minimum of $\hat{J}:\overline{\cF}_0\rightarrow\R$ is achieved at $f_*$. 

Now there are two cases to consider: $f_*\in\cF_0$ and $f_*\in\overline{\cF}_0\setminus\cF_0$. 
In the first case $f_*\in\cF_0$ one can check that $f_*$ is a critical point of $\hat{J}$ considered as a map from all of $L^2(0,1)$ to $\R$. 

It is clear that $\hat{J}:L^2(0,1)\rightarrow\R$ is smooth and a quick computation shows that the gradient is given by 
\begin{align*}
\nabla \hat{J}(f)=(2-3S(f))f-2\id,
\end{align*}
where
\[
S(f):=\ska{f-\id,f+\id}_{L^2(0,1)}=\norm{f}_{L^2(0,1)}^2-\frac{1}{3}.
\]
Thus for a critical point of $\hat{J}$ there holds $S(f)\neq \frac{3}{2}$ and
\[
f=\frac{2}{2-3S(f)}\id.
\]
Plugging this identity into the definition of $S(f)$ one obtains that
\begin{align*}
S(f)=\frac{4}{(2-3S(f))^2}\norm{\id}^2_{L^2(0,1)}-\frac{1}{3}
\end{align*}
or equivalently $S(f)\in\set{0,1}$. Thus $\hat{J}:L^2(0,1)\rightarrow \R$ has exactly two critical points in $f=\id$ and $f=-2\id$, and only $f=\id$ is contained in $\cF_0$. Consequently if the minimum of $\hat{J}_{|\overline{\cF}_0}$ is achieved at $f_*\in\cF_0$, then $f_*=\id$ and $\hat{J}_{|\overline{\cF}_0}\geq \hat{J}(\id)=0$.

If we assume that $\id$ is not minimizing $\hat{J}_{|\overline{\cF}_0}$, then any minimizer $f_*$ lies in $\overline{\cF}_0\setminus\cF_0$ and satisfies $\hat{J}(f_*)<0$. Without loss of generality we can assume $f_*\geq 0$ and $f_*$ to be nondecreasing, otherwise we replace $f_*$ by the monotone increasing rearrangement of $\abs{f_*}$, which only decreases $\hat{J}$, cf. \eqref{eq:form_of_hat_J}.
These two properties together with $f_*\in\overline{\cF}_0\setminus\cF_0$ imply that there exist $f_0\in\cF_0$ and $a\in[0,1)$, such that
\begin{align*}
f_*(y)=f_0\left(\frac{y}{a}\right)\cX_{(0,a)}(y)+\cX_{(a,1)}(y)
\end{align*}
for a.e. $y\in(0,1)$. Here $\cX$ denotes the indicator function and for $a=0$ this expression is understood as $f_*=\cX_{(0,1)}$. In a straightforward way one sees that
\begin{gather*}
\norm{f_*}_{L^2(0,1)}^2=a\norm{f_0}_{L^2(0,1)}^2+1-a,\\
\ska{\id,f_*}_{L^2(0,1)}=a^2\ska{\id,f_0}_{L^2(0,1)}+\frac{1}{2}(1-a^2),
\end{gather*}
such that 
\[
\hat{J}(f_*)=a^2\hat{J}(f_0).
\]
Since by assumption $\hat{J}(f_*)<0$, this equality implies $a\in(0,1)$ and $\hat{J}(f_0)<\hat{J}(f_*)$, which tells us that $f_*$ can not be a minimizer of $\hat{J}_{|\overline{\cF}_0}$. Due to this contradiction we conclude that the infimum of $\hat{J}_{|\overline{\cF}_0}$ is achieved at $f_*=\id$.

Finally the strict inequality $\hat{J}(f)>0$ for $f\in\cF_0\setminus\{\id\}$ follows from the fact that $\id$ is the only critical point of $\hat{J}:L^2(0,1)\rightarrow \R$ lying in $\cF_0$.
\end{proof}

\subsection{Beyond small-time behaviour}\label{sec:beyond_small_times}

While the subsolution constructed in the previous subsection focused on minimizing the initial energy dissipation, one could also be interested in the long-time behaviour of such subsolutions. In particular, how can the subsolution be continued after $a$ reaches $L$, i.e. the mixing zone touches the upper boundary. There are two long-time states which are of interest, 
namely the one where both the density and the momentum are vanishing everywhere (hence there are no longer two different density fluids, but only one completely mixed fluid), and the configuration $-\rho_0$, where the higher density fluid occupies the lower half of the domain, respectively the lower density fluid occupies the upper half (i.e. gravity demixes the two fluids in the long term). We will show that both of these configurations can be achieved.

\subsubsection{Converging towards the fully mixed, isotropic state}

\begin{proof}[Proof of Proposition \ref{prop:mix}]
We claim that one may extend (in an admissible way) the subsolution given in Example \ref{ex:ss} from $\Omega\times \left(0,\sqrt{\frac{3L}{gA}}\right)$ to $\mathscr{D}:=\Omega\times (0,+\infty)$ simply by considering for $(x,t)\in(0,1)^{n-1}\times(-L,L)\times \left(\sqrt{\frac{3L}{gA}},+\infty\right)$ the following:
\begin{align*}
\rho(x,t)=\frac{3x_n}{gAt^2},\quad m(x,t)=\frac{3}{gAt^3}(x_n^2-L^2)e_n,\quad \varepsilon=\varepsilon(t)=\frac{2}{3n} \sqrt{\frac{3L}{gA}}\frac{1}{t},
\end{align*}
 $v\equiv 0$, 
$\sigma,p,\tilde e$ as in \eqref{eq:subsp}, \eqref{eq:subse}, as well as the mixing zone $$\mathscr{U}:=\set{(x,t)\in\mathscr{D}:|x_n|<\frac{gAt^2}{3}}.$$ Indeed, one observes through a straightforward calculation that for this choice, the maximum in \eqref{eq:subse} is always achieved for the first term if $x_n\geq 0$ and $t\geq \sqrt{\frac{3L}{gA}}$, i.e.
\begin{multline*}
\frac{m_n(x,t)^2}{n(1+\rho(x,t))^2}+\varepsilon(t) gAx_n\geq \frac{m_n(x,t)^2}{n(1-\rho(x,t))^2}-\varepsilon(t) gAx_n \\
\Leftrightarrow 2n\varepsilon(t) gA x_n \geq \frac{9}{g^2A^2t^6} \frac{4\rho(x,t)(x_n^2-L^2)^2}{\left(1-\rho(x,t)^2\right)^2}=\frac{27}{g^3A^3t^8} \frac{4L^4x_n(1-(x_n/L)^2)^2}{\left(1-\rho(x,t)^2\right)^2},
\end{multline*}
$\text{which follows if }
\frac{n}{2}\varepsilon(t) \geq \frac{1}{3} \left( \frac{3L}{gA t^2}\right)^4,$ by observing that $\left|\frac{1-(x_n/L)^2}{1-\rho(x,t)^2}\right|<1$. Plugging in the value for $\varepsilon(t)$, this is equivalent to $1\geq \left(\sqrt{\frac{3L}{gA}}\frac{1}{t}\right)^7$, which is obviously true for $t\geq \sqrt{\frac{3L}{gA}}$.

Hence we have
$$\tilde{e}_\varepsilon(x,t)=\frac{1}{n}\frac{(x_n^2-L^2)^2}{t^2(\frac{gA}{3}t^2+x_n)^2}+\varepsilon(t)gA x_n +(1-\rho(x,t)^2)\delta(x,t)\text{ for }x_n\geq 0,\ t\geq \sqrt{\frac{3L}{gA}}.$$
Then, recalling \eqref{eq:subste} and using the parity of $x_n\mapsto\tilde e_\varepsilon(x,t)$ as well as $x_n\mapsto x_n\rho(x,t)$, for $t>\sqrt{\frac{3L}{gA}}$ one obtains that 
\begin{align*}
E_{f,a,\varepsilon}(t)
=& 2\int_0^L \frac{n}{2}\left(\tilde{e}_\varepsilon(x,t)-(1-\rho(x,t)^2)\delta(x,t)\right)+\left(1-\frac{n}{2}\varepsilon\right)gA x_n\rho\:dx_n
\\=& \int_0^L\frac{(y^2-L^2)^2}{t^2(\frac{gA}{3}t^2+y)^2}+\frac{2}{3} \sqrt{\frac{3L}{gA}}\frac{1}{t}gA  y +2 \frac{3}{t^2} \left(1-\frac{1}{3}\sqrt{\frac{3L}{gA}}\frac{1}{t}\right) y^2\, dy\\=&\frac{1}{t^2}\left(\int_0^L \frac{(y^2-L^2)^2}{(\frac{gA}{3}t^2+y)^2}\, dy +2L^3\left(1-\frac{1}{3}\sqrt{\frac{3L}{gA}}\frac{1}{t}\right)\right) +\frac{gA L^2}{3}\sqrt{\frac{3L}{gA}}\frac{1}{t},
\end{align*}
which is decreasing with respect to $t,$ since
\begin{align*}
\frac{d}{dt}\left(\frac{1}{t^2}\left(1-\frac{1}{3}\sqrt{\frac{3L}{gA}}\frac{1}{t}\right)\right)=-\frac{2}{t^3}\left(1-\frac{1}{2}\sqrt{\frac{3L}{gA}}\frac{1}{t}\right)<0,\text{ for }t\geq \sqrt{\frac{3L}{gA}}.
\end{align*}
Therefore, the admissibility follows. 

To conclude the proof of Proposition \ref{prop:mix}, observe that the limit of the subsolution as $t\to+\infty$ is identically zero, and $\delta$ can be chosen such that the energy of the system also decays to zero in the limit at $+\infty$.
\end{proof}

\begin{remark}
Since the kinetic energy of the solutions associated with the constructed subsolution goes to $0$ as $t\rightarrow+\infty$, any turbulent motion, in fact any motion, will vanish as $t\to+\infty$.
Note that one could have made the same construction while keeping $\varepsilon=\frac{2}{3n}$ constant and still have obtained an admissible subsolution. However, the associated energy as $t\to+\infty$ would not vanish, which would imply that there is still some turbulence at infinite time.
\end{remark}

\subsubsection{Demixing in finite time}

Let us now construct an example of a different admissible continuation past the time when the mixing zone touches the upper boundary, one where first the density profile is rotated by $180$ degrees, and then the mixing zone shrinks until the stable configuration $-\rho_0$ is reached. 

\begin{proof}[Proof of Proposition \ref{prop:demix}]
We will do this in two steps. 

\textbf{Step 1: Rotation.}

Denote $T_0:=\sqrt{\frac{3L}{gA}}$.
As before, on $\left[0,T_0\right]$ we consider the subsolution given in Example \ref{ex:ss}. We claim that there exist $\tilde T>T_0$ and a non-increasing, continuously differentiable function $r:\left[T_0,\tilde T\right]\to\left[-\frac{1}{L},\frac{1}{L}\right]$
satisfying $r\left(T_0\right)=\frac{1}{L}$, $r(\tilde T)=-\frac{1}{L}$, $\dot{r}\left(T_0\right)=-2\sqrt{\frac{gA}{3L^3}}$,
 such that setting
\begin{align*}
\rho(x,t)=r(t)x_n,\quad m(x,t)=-\frac{\dot{r}(t)}{2}(x_n^2-L^2) e_n,
\end{align*}
as well as $v\equiv 0$, $\varepsilon=\frac{2}{3n},$
and $\sigma,p,\tilde e_\varepsilon$ as in \eqref{eq:subsp}, \eqref{eq:subse}, for $(x,t)\in(0,1)^{n-1}\times(-L,L)\times (T_0,\tilde T)\subset\mathscr{U}$,
yields a subsolution which is continuous, piecewise $\cC^1$ and admissible on $\Omega\times(0,\tilde T]$.

Indeed, the continuity at $t=T_0$ follows from the definitions of $r(T_0)$ and $\dot{r}(T_0)$.
To check the admissibility, one needs to treat the maximum in \eqref{eq:subse}. Once again, through simple calculations one obtains for $x_n\geq 0$, $t\in[T_0,\tilde{T}]$ that if $r(t)\dot{r}(t)^2\leq\frac{4gA}{3L^4}$, then the maximum is realized by the first term, i.e.
$$\tilde{e}_\varepsilon(x,t)=\frac{1}{n}\left(\frac{\dot{r}(t)^2(x_n^2-L^2)^2}{4(1+r(t)x_n)^2}+\frac{2}{3}gA x_n \right)+(1-\rho(x,t)^2)\delta(x,t).$$
Using once more the parity of $x_n\mapsto\tilde e_\varepsilon(x,t)$ and $x_n\mapsto x_n\rho(x,t)$, one obtains that in this case the corrected total energy at time $t\in[T_0,\tilde{T}]$ reads
\begin{align}\label{eq:subser}
\begin{split}
E_{r}(t):=&\int_\Omega\frac{n}{2}(\tilde{e}_\varepsilon(x,t)-(1-\rho(x,t)^2)\delta(x,t))+\left(1-\frac{n}{2}\varepsilon\right)gA x_n\rho\:dx\\=&\frac{\dot{r}(t)^2}{4}I(r(t)) +\frac{4}{9}gAL^3r(t)+\frac{1}{3}gAL^2,
\end{split}
\end{align}
where
\begin{align*}
I(r)&:=\int_0^L\frac{(y^2-L^2)^2}{(1+ry)^2} \, dy
\end{align*}

Let us now construct a function $r$ satisfying the properties stated above. 

Observe that $I\in\cC^1\left(\left(-\frac{1}{L},+\infty\right)\right)\cap\cC^0\left(\left[-\frac{1}{L},+\infty\right)\right)$ with $I\left(-\frac{1}{L}\right)=\frac{7}{3}L^5$. Moreover, $I$ is clearly positive and monotone decreasing on the intervall $\left[-\frac{1}{L},+\infty\right)$.
Let $r:[T_0,T_{\max})\rightarrow\R$ be the unique solution of the initial value problem
\begin{align}\label{eq:err}
\dot{r}(t)=-2\sqrt{\frac{gAL^2}{9I(r(t))}},\quad r(t)\in\left(-\frac{1}{L},+\infty\right),\quad r(T_0)=\frac{1}{L},
\end{align}
where $T_{\max}$ denotes the maximal time of existence of the solution. 

We claim that $T_{\max}<+\infty$ and $r$ as a function extends continuously to $[T_0,T_{\max}]$ with $r(T_{\max})=-\frac{1}{L}$. Assume to the contrary that $T_{\max}=+\infty$, then $r(t)>-\frac{1}{L}$ for all $t\geq T_0$.
But now integrating \eqref{eq:err} for $t\in(T_0,T_{\max})$ and using that $I$ is decreasing, one has the contradiction
\begin{align*}
-\frac{2}{L}< r(t)-\frac{1}{L}=-2\int_{T_0}^t \sqrt{\frac{gAL^2}{9I(r(s))}} \, ds \leq -2\sqrt{\frac{gAL^2}{9I(-\frac{1}{L})}}(t-T_0)\rightarrow -\infty
\end{align*}
as $t\rightarrow +\infty$. Hence $T_{\max}<+\infty$ and then necessarily $\lim_{t\rightarrow T_{\max}}r(t)=-\frac{1}{L}$, because the orbit $r([T_0,T_{\max}))$ is bounded from above due to the monotonicity of $r$. We therefore set $\tilde{T}:=T_{\max}$.

Next due to $I\left(\frac{1}{L}\right)=\frac{1}{3}L^5$ and \eqref{eq:err} it is easy to see that $\dot{r}\left(T_0\right)=-2\sqrt{\frac{gA}{3L^3}}$. 

Finally, let us show that the associated   corrected total energy function $E_r$ is decreasing, to conclude the admissibility on $[T_0,\tilde T]$ of our subsolution.
For this, we first show that one has $r(t)\dot{r}(t)^2\leq\frac{4gA}{3L^4}$, so that in $\tilde{e}_\varepsilon$ indeed the first term under the maximum is selected for $x_n\geq 0$ and thus \eqref{eq:subser} holds. Once again, this follows from \eqref{eq:err} by using the monotonicity of $I$ and $r$:
\begin{align*}
r(t)\dot{r}(t)^2\leq\frac{1}{L}\frac{4}{I(\frac{1}{L})}\frac{1}{9}gAL^2=\frac{12}{L^6}\frac{1}{9}gAL^2=\frac{4gA}{3L^4}.
\end{align*}
Since the  corrected total energy function $E_r$ is then given by formula \eqref{eq:subser}, we may plug \eqref{eq:err} into \eqref{eq:subser} to further obtain that
\begin{align*}
E_r(t)=\frac{1}{9}gAL^2+\frac{4}{9}gAL^3r(t)+\frac{1}{3}gAL^2=\frac{4}{9}gAL^2\left(1+Lr(t)\right),
\end{align*}
which is clearly decreasing since $r$ is decreasing. This concludes the construction for the rotation of the profile.

\textbf{Step 2: Shrinking of the mixing zone.}

We will now further extend the subsolution constructed above past the time $\tilde T$. 
Let
$$T_{end}:=\tilde T+\sqrt{\frac{21L}{gA}},$$
and set $\mathscr{D}:=\Omega\times(0,T_{end})$, $\mathscr{U}:=\set{(x,t)\in\mathscr{D}:x_n\in(-a(t),a(t))}$, with 
\begin{align*}
a(t)= \left\{
\begin{array}{ll}
      \frac{gAt^2}{3}, & 0\leq t \leq T_0 \\
      L, & T_0\leq t\leq \tilde T\\
     \frac{gA(t-T_{end})^2}{21}, & \tilde T\leq t\leq T_{end} \\
\end{array} 
\right. .
\end{align*}

On $[0,T_0]$ our subsolution will coincide with the one from Example \ref{ex:ss}, on $[T_0,\tilde T]$ with the one constructed in Step 1, and on $[\tilde T, T_{end}]$ it will be of the form
\begin{align*}
\rho(x,t)=-\frac{x_n}{a(t)},\quad m(x,t)=\frac{\dot{a}(t)}{2}\left(1-\frac{x_n^2}{a(t)^2}\right)e_n,
\end{align*}
 $v\equiv 0$, $\varepsilon=\frac{2}{3n},$
$\sigma,p,\tilde e$ as in \eqref{eq:subsp}, \eqref{eq:subse}, for $(x,t)\in\mathscr{U}$. Outside the mixing zone we consider $\rho=-\rho_0$, $v\equiv 0$ and $\tilde{e}_\varepsilon(x,t)=-\varepsilon gA \abs{x_n}$.

One can check through straightforward calculations that this choice  makes $\rho$ and $m$ (and hence the whole subsolution) continuous at $t=\tilde T$. 

Clearly at $T_{end}$, this subsolution reaches the stable configuration 
$\rho=-\rho_0$, $v\equiv 0$ with no mixing. 
All that remains to be checked is the admissibility on $[\tilde T,T_{end}]$.

Once more, one may easily evaluate the maximum in \eqref{eq:subse} to obtain that for $x_n\geq 0$ one has
$$\tilde{e}_\varepsilon(x,t)=\frac{m_n(x,t)^2}{n(1+\rho(x,t))^2}+\varepsilon gAx_n+\left(1-\rho(x,t)^2\right)\delta(x,t).$$
On the other hand, using once more the parity of $x_n\mapsto\tilde{e}_\varepsilon(x,t)$, plugging in the formulas for $\rho$ and $m$, and using the change of variables $y=\frac{x_n}{a(t)}$, we have
\begin{align*}
\int_\Omega&\frac{n}{2}(\tilde{e}_\varepsilon(x,t)-(1-\rho(x,t)^2)\delta(x,t))+\left(1-\frac{n}{2}\varepsilon\right)gA x_n\rho\:dx
\\&=2\int_0^L \frac{n}{2}\left(\tilde{e}_\varepsilon(x,t)-(1-\rho(x,t)^2)\delta(x,t)\right)+\left(1-\frac{n}{2}\varepsilon\right)gA x_n\rho\:dx_n
\\&=\int_0^{a(t)}\frac{m_n(x,t)^2}{(1+\rho(x,t))^2}+\frac{2}{3} gAx_n-\frac{4}{3}gA\frac{x_n^2}{a(t)}\, dx_n
+2\int_{a(t)}^L \left(-\frac{1}{3}-\frac{2}{3} \right)gA x_n \, dx_n 
\\&=a(t)\int_0^1\frac{\dot{a}(t)^2}{4}(1+y)^2+\frac{2}{3}gAa(t)y-\frac{4}{3}gAa(t)y^2 \, dy-gA(L^2-a(t)^2)
\\&=\frac{7}{12}a(t)\dot{a}(t)^2+\frac{8}{9}gAa(t)^2-gAL^2,
\end{align*}
which is clearly decreasing on $[\tilde{T},T_{end}]$ since both $a$ and $\abs{\dot{a}}$ are decreasing. This concludes the proof of Proposition \ref{prop:demix}.
\end{proof}

\vspace{20pt}
\noindent Mathematisches Institut,  Universit\"at Leipzig,  Augustusplatz 10, D-04109 Leipzig \\
\texttt{bjoern.gebhard@math.uni-leipzig.de}\\
\texttt{jozsef.kolumban@math.uni-leipzig.de}
\end{document}